\documentclass{article}
\usepackage{times}
\usepackage{graphicx}
\usepackage{wrapfig}
\usepackage{amsmath,amsthm}
\usepackage{amssymb}
\usepackage{dsfont}
\usepackage{xspace}
\usepackage{color}
\usepackage{enumitem}
\usepackage[linesnumbered]{algorithm2e}
\usepackage{subfigure}
\usepackage{tikz}
\usepackage{lscape}
\usepackage[hypertexnames=false]{hyperref}
\usepackage[left=3cm,right=3cm,top=3cm,bottom=3.5cm]{geometry}

\newcommand{\interpolimage}{\mathcal{U}} 
\newcommand{\imagecurve}{u} 
\newcommand{\image}{u} 
\newcommand{\imagevec}{\mathbf{u}} 
\newcommand{\imageinnervec}{{ \mathbf{\hat u}}} 
\newcommand{\imageinnervectest}{{ \mathbf{\hat v}}} 
\newcommand{\defvec}{\mathbf{\Phi}}

\newcommand{\WEner}{{\mathcal{W}}}
\newcommand{\Image}{U} 
\newcommand{\basisfct}{\Theta}

\newcommand{\Right}{{\mathbf{R}}}

\newcommand{\U}{{\mathbf{U}}}

\newcommand{\Matrix}{A}
\newcommand{\MatrixBF}{\mathbf{A}}

\newcommand{\ESPhi}[1]{{\mathbf{E}^D_{#1}}}

\newcommand{\V}{{\mathcal{V}}}

\newcommand{\W}{{\mathcal{W}^D}}

\newcommand{\set}[1]{{\{{#1}\}}}
\newcommand{\setof}[2]{{\{{#1}\,|\,{#2}\}}}
\newcommand{\metric}{{g}}
\renewcommand{\div}{{\mathrm{div}}}
\newcommand{\R}{{\mathbb{R}}}
\newcommand{\N}{{\mathbb{N}}}
\DeclareMathOperator{\argmin}{argmin}

\renewcommand{\d}{{\,\mathrm{d}}}

\newcommand{\tr}{{\mathrm{tr}}}
\newcommand{\Id}{{\mathds{1}}}

\newcommand{\compdom}{{D}}
\newcommand{\Ih}{{\mathcal{I}_h}}
\newcommand{\mass}{{\mathbf{M}}}
\newcommand{\stiff}{{\mathbf{S}}}
\newcommand\restr[2]{{\left.\kern-\nulldelimiterspace #1 \vphantom{\big|}\right|_{#2}}}

\newcommand{\funcnorm}[2]{\left\|#1\right\|_{#2}}

\newcommand{\pathenergy}{{\mathbf{E}}}
\newcommand{\continuouspathenergy}{\mathcal{E}}
\newcommand{\imagespace}{L^2(D)}
\newcommand{\admset}{\mathcal{A}}

\newcommand{\notinclude}[1]{}

\def\Xint#1{\mathchoice
{\XXint\displaystyle\textstyle{#1}}
{\XXint\textstyle\scriptstyle{#1}}
{\XXint\scriptstyle\scriptscriptstyle{#1}}
{\XXint\scriptscriptstyle\scriptscriptstyle{#1}}
\!\int}
\def\XXint#1#2#3{{\setbox0=\hbox{$#1{#2#3}{\int}$}
\vcenter{\hbox{$#2#3$}}\kern-.5\wd0}}
\def\dashint{\Xint-}

\newcommand{\beq}{\begin{equation*}}
\newcommand{\eeq}{\end{equation*}}
\newcommand{\beqn}{\begin{equation}}
\newcommand{\eeqn}{\end{equation}}
\newcommand{\beqa}{\begin{eqnarray*}}
\newcommand{\eeqa}{\end{eqnarray*}}
\newcommand{\beqan}{\begin{eqnarray}}
\newcommand{\eeqan}{\end{eqnarray}}

\newcounter{counter}

\makeatletter
\DeclareRobustCommand\onedot{\futurelet\@let@token\@onedot}
\def\@onedot{\ifx\@let@token.\else.\null\fi\xspace}

\def\eg{\emph{e.g}\onedot} 
\def\ie{\emph{i.e}\onedot} 
\def\cf{\emph{cf}\onedot}

\makeatother

\theoremstyle{plain}\newtheorem{theorem}{Theorem}[section]

\newtheorem{proposition}[theorem]{Proposition}

\theoremstyle{definition}
\newtheorem{definition}{Definition}[section]

\theoremstyle{remark}
\newtheorem*{remark}{Remark}

\begin{document}

\title{Time Discrete Geodesic Paths in the Space of Images}
\author{B. Berkels, A. Effland, M. Rumpf}
\maketitle

\begin{abstract}
In this paper the space of images is considered as a Riemannian manifold using the metamorphosis approach 
\cite{MiYo01,TrYo05a,TrYo05}, where the underlying Riemannian metric simultaneously measures
the cost of image transport and intensity variation. 
A robust and effective variational time discretization of geodesics paths is proposed. 
This requires to minimize a discrete path energy consisting of a sum of consecutive image matching functionals 
over a set of image intensity maps and pairwise matching deformations.
For square-integrable input images the existence of discrete, connecting geodesic paths 
defined as minimizers of this variational problem is shown. Furthermore, 
$\Gamma$-convergence of the underlying discrete path energy to the continuous path energy is proved.
This includes a diffeomorphism property for the induced transport and the existence of a
square-integrable weak material derivative in space and time.
A spatial discretization via finite elements combined with an alternating descent scheme in the set of image intensity maps and the set of matching deformations is presented to approximate discrete geodesic paths numerically.
Computational results underline the efficiency of the proposed approach
and demonstrate important qualitative properties.
\end{abstract}

\section{Introduction}\label{sec:intro}
The study of spaces of shapes from the perspective of a Riemannian manifold  allows to transfer many important  
concepts from classical geometry to these usually infinite-dimensional spaces.
During the past decade, this Riemannian approach had an increasing impact on the development of new methods in 
computer vision and imaging, ranging from shape morphing and modeling, \eg \cite{KiMiPo07}, and shape statistics, \eg 
\cite{FlLuPi04}, to computational anatomy \cite{BeMiTrYo02}. A variety of Riemannian shape spaces has been investigated 
in the literature. Some of them are finite-dimensional and consider polygonal curves or triangulated surfaces as shapes 
\cite{KiMiPo07,LiShDi10}, but most approaches deal with infinite-dimensional spaces of shapes. Prominent examples with a full-fledged geometric theory are spaces of planar curves with curvature-based metric \cite{MiMu04}, 
elastic metric \cite{SrJaJo06} or Sobolev-type metric \cite{ChKePo05,MiMu07,SuYeMe07}.
The concept of optimal transport was used to study the space of images, where image intensity functions are considered as probability measures, \eg Zhang et al. \cite{ZhYaHa07} minimize the 
Monge-Kantorovich functional $\int_{\!\compdom}\!|\psi(x)\!-\!x|^2\rho_0(x)\d x$ over all mass preserving
mappings $\psi\!:\!\!\compdom\!\!\to\!\!\compdom$.
Benamou and Brenier \cite{BeBr00} used a flow reformulation of optimal transport, which nicely fits into the Riemannian context.

For only a few nontrivial application-oriented Riemannian spaces geodesic paths can be computed in closed form (\eg \cite{YoMiSh08,SuMeSo11}), else the system of geodesic ODEs has to be solved using numerical time stepping schemes (\eg \cite{KlSrMi04,BeMiTr05}).
Alternatively, geodesic paths connecting shapes 
can also be approximated via the minimization of discretized path length \cite{ScClCr06}
or path energy functionals \cite{FuJuScYa09,WiBaRu10}. In this paper, we will develop such a variational time discretization on the space of images using the metamorphosis approach proposed by Trouv\'e and Younes \cite{TrYo05,TrYo05a,HoTrYo09}. This approach is a generalization of the flow of diffeomorphism approach 
initiated by Dupuis, Grenander and Miller \cite{DuGrMi98}. 

The concept of variational time discretization is a powerful tool  in the discretization of gradient flows
and for Hamiltonian mechanical systems. The analog of the time discrete path energy considered here is a discrete action sum.
For a historic account we refer to  \cite{HaLuWa06}.  Numerical analysis was exploited 
from the $\Gamma$-convergence perspective in \cite{MuOr04}, and from the ODE-discretization perspective under the name of variational integrators in \cite{LeMaOr04,ObJuMa10}. Thereby, the time continuous Lagrangian on some time interval is
replaced by a time discrete functional  related to our functional $\WEner$ and defined directly on configuration variables and not involving momentum variables. 

Instead of discretizing the underlying flow and incorporating the target configuration at the end time via a constraint, the variational discretization is based on
the direct minimization of a discrete path energy subject to data given at the initial and the end time.  This approach turned out to be very stable and robust, and 
even for very small numbers of time steps one obtains qualitatively good results. Furthermore, proceeding from coarse to fine time discretization, an efficient  
cascadic  minimization strategy can be implemented. 
In the context of shape spaces, this concept has already been used 
in the space of viscous objects  \cite{FuJuScYa09,WiBaRu10, RuWi12}, but without a rigorous mathematical foundation. In \cite{RuWi12b}, a discrete geodesic calculus on finite- and on certain infinite-dimensional shape spaces with the structure of a Hilbert manifolds was developed and a full-fledged convergence analysis could be established. This theory immediately applies for instance to the (finite-dimensional) Riemannian manifold of discrete shells \cite{HeRuWa12,HeRuSc14}. In this paper, we expand part of this theory to the metamorphosis model, which lacks a Hilbert manifold structure. 

In what follows, we will briefly review both the flow of diffeomorphism and the metamorphism approaches
as a basis for the discussion of our time discrete metamorphosis model and the $\Gamma$-convergence analysis to be presented in this paper.

\paragraph{Flow of diffeomorphism} Here, we give a very short exposition and refer to \cite{DuGrMi98,BeMiTr05,JoMi00,MiTrYo02} for more details.
Following the classical paradigm by Arnold \cite{Ar66a,ArKh98}, one studies the temporal change of image intensities from the perspective of a family of diffeomorphisms $(\psi(t))_{t\in [0,1]}: \bar D \to \R^d$ on the closure of the image domain $D\subset \R^d$ for $d=2,3$ describing a flow, which transports image intensities along particle paths. In what follows, we suppose that $D$ is a bounded domain with Lipschitz boundary. 
A path energy 
\beq
\pathenergy[(\psi(t))_{t\in [0,1]}] = \int^1_0\int_D L[v(t),v(t)] \d x \d t
\eeq
is associated which each path $(\psi(t))_{t\in [0,1]}$ in the space of images, where $v(t) = \dot \psi(t) \circ \psi^{-1}(t)$ represents the Eulerian velocity of the underlying flow and $L$ is a quadratic form corresponding to a higher order elliptic operator. 
Physically, the metric $g_{\psi(t)}(\dot \psi(t),\dot \psi(t)) = \int_D L[v(t),v(t)] \d x$ describes the viscous dissipation in a multipolar fluid model as investigated by Ne\v{c}as and \v{S}ilhav\'y \cite{NeSi91}.
From this perspective, a suitable choice for the viscous dissipation is given by a combination of a classical Newtonian flow and a simple multipolar dissipation model, namely
\beqn
L[v(t),v(t)] :=  \tfrac{\lambda}{2} (\tr\varepsilon[v])^2+ \mu\tr(\varepsilon[v]^2) + \gamma |D^m v|^2\,,
\label{definitionEllipticOperator}
\eeqn
where $\varepsilon[v]= \frac12(\nabla v +  \nabla v^T)$, $m>1+\frac{d}{2}$ and $\lambda, \, \mu,\, \gamma >0$ (throughout this paper gradient $\nabla$, divergence $\div$, and higher order derivatives $D^m$ are always evaluated with respect to the spatial variables). The first two terms of the integrand represent the usual dissipation density in a Newtonian fluid, whereas the third term represents a higher order measure for friction. Under suitable assumptions on $L$ it is shown in \cite[Theorem 2.5]{DuGrMi98} that paths of finite energy, which connect two diffeomorphisms $\psi(0)=\psi_A$ and $\psi(1)=\psi_B$, are indeed one-parameter families of diffeomorphisms. Furthermore, for any minimizing sequence of paths a subsequence converges uniformly to an energy minimizing path, in particular the minimizing path solves $\dot \psi(t,\cdot) = v(t,\psi(t,\cdot))$ for every $t\in [0,1]$, where $v$ is the energy minimizing velocity
(cf. \cite[Theorem 3.1]{DuGrMi98}).
Given two image intensity functions 
$\image_A,\image_B\in\imagespace$, an associated geodesic path is a family of images $\imagecurve= (\imagecurve(t): D \to \R)_{t\in [0,1]}$ with 
$\imagecurve(0) = \image_A$ and $\imagecurve(1) = \image_B$, which minimizes the path energy. 
The associated flow of images is given by $u(t) = \image_A \circ \psi^{-1}(t)$.
In medical applications \cite{BeMiTrYo02}, the diffeomorphisms represent deformations of anatomic reference structures described by some image $\image_A$. Thus, each diffeomorphism $\psi(t):\bar D \to \R^d$ for $t\in [0,1]$ represents a particular anatomic configuration or shape of these structures.
Let us remark that this model is obviously invariant under rigid body motions, \ie rigid body motions are generated by motion fields $v$ with spatially constant, skew symmetric Jacobian, for which $\varepsilon[v]=0$ and $D^m v=0$.

\paragraph{Metamorphosis}
The metamorphosis approach was first proposed by  Miller and Younes \cite{MiYo01} and comprehensively analyzed by Trouv\'e and Younes \cite{TrYo05}. 
It allows in addition for image intensity variations along motion paths. Conceptually and under the assumption that the family of images $u$ is sufficiently smooth, the associated metric for some parameter $\delta >0$ can be written as 
\beq
\metric(\dot \imagecurve, \dot \imagecurve) = \min_{v:\bar D \to \R^d} \int_D L[v,v]  + \frac1\delta (\dot \imagecurve + \nabla \imagecurve \cdot v)^2 \d x
\eeq
and induces the path energy $\pathenergy[\imagecurve] = \int_0^1 \metric(\dot \imagecurve(t), \dot \imagecurve(t)) \d t\,$.
Let $\frac{D}{\partial t} \imagecurve = \dot \imagecurve + \nabla \imagecurve \cdot v$ denote the material derivative of $\imagecurve$.  Obviously, the same temporal change $\dot \imagecurve(t)$ in the image intensity can be implied by different motion fields $v(t)$ and different associated material derivatives $\frac{D}{\partial t} \imagecurve$, \ie $\dot \imagecurve(t) = \frac{D}{\partial t} \imagecurve - \nabla \imagecurve \cdot v$. 
In fact, one introduces a nonlinear geometric structure on the space of images by considering 
equivalence classes of pairs $(v,\frac{D}{\partial t} \imagecurve)$ as tangent vectors in the space of images,
where such pairs are supposed to be equivalent iff they imply the same temporal change $\dot \imagecurve$.
Hence, to evaluate the metric on such tangent vectors one has to minimize over the elements of the equivalence class and computing a geodesic path requires to optimize both the temporal change of the image intensity and the motion field. Thereby, the first term $L[v,v]$ reflects the cost of the underlying transport
and the term $\frac1\delta (\frac{D}{\partial t} \imagecurve)^2$ penalizes the variation of the image intensity along motion paths. 

However, typically images are not smooth and paths in image space are neither smooth in time nor in space.
Thus, the classical notion of the material derivative $\dot \imagecurve + \nabla \imagecurve \cdot v$ is not well-defined. In \cite{TrYo05a} Trouv\'e and Younes established a suitable  generalization of the above nonlinear geometric structure on $L^2(D):=L^2(D, \R)$, which is used as the space of images,
based on a proper notion of weak material derivatives. 
Here, we recall the fundamental ingredients of this approach. In fact, for  $v \in L^2((0,1),W^{m,2}(D,\R^d)\cap W^{1,2}_0(D,\R^d))$ the function $z \in L^2((0,1),L^2(D))$ is defined as a weak material derivative of a function $\imagecurve \in L^2((0,1),L^2(D))$ if 
\begin{equation}
\int_0^1 \int_D \eta z \d x \d t = -  \int_0^1 \int_D (\partial_t \eta + \div (v \eta)) \imagecurve \d x \d t
\label{eq:definitionMaterialDerivative}
\end{equation}
for $\eta\in C^{\infty}_c((0,1)\times D)$. Here, $W^{m,2}$ denotes the usual Sobolev space of functions with square-integrable derivatives up to order $m$, and $W^{1,2}_0$ is the 
space of functions in $W^{1,2}$ with vanishing trace on the boundary. 
In terms of Riemannian manifolds, Trouv\'e and Younes equipped
the space of images $L^2(D)$ with the following nonlinear structure: 
Let
\[
N_{\image}=\left\{ w=( v,z)\in W:\ \int_D z \eta + \image \div(\eta v) \d x =0\ \forall\eta\in C^{\infty}_c(D)\right\}\,.
\]
For $W = (W^{m,2}(D,\R^d)\cap W^{1,2}_0(D,\R^d)) \times L^2(D)$ 
the tangent space at $\image \in L^2(D)$ is defined as 
$T_\image L^2(D)= \{\image\}\times W / N_{\image}$ and elements in this tangent space, which are equivalence classes, are denoted by $(u,\overline{(v,z)})$.
The tangent bundle is given by
\beq
T L^2(D)=\bigcup_{\image\in L^2(D)} T_{\image} L^2(D).
\eeq
Furthermore, let $\pi(u,\overline{(v,z)}) = u$ be the projection onto the image manifold.
Indeed, this is a weak formulation of the above notion of a tangent space as an equivalence class.
Following the usual Riemannian manifold paradigm, a curve $\imagecurve \in C^0([0,1],L^2(D))$ in the space of images is called \textit{continuously differentiable},
iff there is a continuous curve $t\mapsto w(t)=(v(t),z(t))$ in $W$ such that for any $\eta\in C^{\infty}_c(D)$ the mapping
$t\mapsto \int_D \imagecurve(t) \eta \d x$ is continuously differentiable (denoted by $\imagecurve \in C^1([0,1],L^2(D))$) and
 \beqn\label{eq:contpath}
\frac{\mathrm{d}}{\mathrm{d} t} \left(\int_D \imagecurve(t) \eta \d x\right) 
= \int_D z(t) \eta + \imagecurve(t) \div(\eta v(t)) \d x\,.
\eeqn
In fact, for a curve $t\to \gamma(t) = \left(\imagecurve(t), \overline{(v(t),z(t))}\right)$  in $T L^2(D)$ the function
$z$ is the (weak) material derivative if \eqref{eq:contpath} holds for all test functions 
$\eta \in C^\infty_c([0,1] \times D)$  and all times $t\in (0,1)$. Furthermore, a curve 
$\imagecurve \in C^0([0,1],L^2(D))$ is defined to be \textit{regular} in the space of images (denoted by $\imagecurve\in H^1((0,1),L^2(D))$), if there exists a measurable path $\gamma:[0,1]\rightarrow TL^2(D)$ with $\pi(\gamma)=\imagecurve$ and bounded $L^2$-norm in space and time, such that
\beqn\label{eq:weakmaterial1}
-\int_0^1 \int_D \imagecurve \partial_t \eta \d x \d t = 
\int_0^1\int_D z \eta + \imagecurve \div(\eta v) \d x \d t
\eeqn
for all $\eta\in C^{\infty}_c((0,1)\times D)$. 
In fact, a continuously differentiable path $\imagecurve \in C^1([0,1],L^2(D))$ is always regular, \ie
$\imagecurve \in  H^1((0,1),L^2(D))$ (cf. \cite[Proposition 4]{TrYo05a}).
Now, for a regular path 
$\imagecurve \in H^1((0,1),L^2(D))$  and for the quadratic form $L[v,v]$ being coercive on 
$W^{m,2}(D,\R^d)\cap W^{1,2}_0(D,\R^d)$ (which can be easily verified for $L$ given in \eqref{definitionEllipticOperator}
using Korn's Lemma) one can rigorously define the path energy
\begin{equation}
\continuouspathenergy[\imagecurve] = \int_0^1 \inf_{\overline{(v,z)}\in T_{\imagecurve(t)} L^2(D)}  \int_D L[v,v]  + \frac1\delta z^2 \d x \d t\,.
\label{eq:DefinitionPathenergy}
\end{equation}
In \cite{TrYo05a}, Trouv\'e and Younes proved the existence of minimizing paths
for given boundary data in time. Adapted to our notion, they have shown that for
$m > 1+\frac{d}{2}$ and $\gamma,\delta>0$ and
given images $\image_A,\image_B\in L^2(D)$ there exists a curve $\imagecurve\in H^1((0,1),L^2(D))$ with $\imagecurve(0)=\image_A$ and
$\imagecurve(1)=\image_B$ such that
\beq
\continuouspathenergy[\imagecurve]=\inf\{\continuouspathenergy[\tilde\imagecurve]:\ \tilde\imagecurve\in H^1((0,1),L^2(D)),\
\tilde\imagecurve(0)=\image_A,\ \tilde\imagecurve(1)=\image_B\}\,.
\eeq
Moreover, the infimum in \eqref{eq:DefinitionPathenergy} is attained for all $t\in [0,1]$, \ie there exist minimizing $(v,z)\in T_{\imagecurve(t)} L^2(D)$.

The proof relies on the observation that $W^{m,2}(D)\cap W^{1,2}_0(D)$ compactly embeds into $C^{1,\alpha}_0(D)$ for $\alpha < m-1-\tfrac{d}{2}$. 
The existence of a geodesic path then follows from \cite[Theorem 6]{TrYo05a}, whereas the addendum is a consequence of \cite[Theorem 2]{TrYo05a}.

\section{The variational time discretization}\label{section:DiscreteModel}
In what follows, we develop a variational approach for the time discretization of geodesic paths 
in the metamorphosis model. This will be based on a time discrete approximation of the above time continuous 
path energy \eqref{eq:DefinitionPathenergy}. 
In what follows, we suppose that $\gamma,\delta>0$, $m>1+\frac{d}{2}$, and define for arbitrary images 
$\image,\,\tilde \image \in \imagespace$ and for a particular energy density $W$ a \textit{discrete energy} 
\begin{equation}
\WEner[\image,\tilde \image]= \min_{\phi\in\admset} \int_{D} W(D \phi) + \gamma |D^m \phi|^2+\frac{1}{\delta} |\tilde \image \circ \phi-\image|^2 \d x\,,
\label{eq:WEnerDefinition}
\end{equation}
where $\admset$ is the \textit{set of admissible deformations}. 
Throughout  this paper, we make the following assumptions with regard to the energy density function $W$:
\setcounter{counter}{0}
\begin{list}{(W\arabic{counter})}{\usecounter{counter}}
\item $W$ is non-negative and polyconvex,
\item $W(A)\geq\beta_0(\det A)^{-s}-\beta_1$ for $\beta_0,\beta_1,s>0$ and every invertible matrix $A$ with $\det A >0$, $W(A) = \infty$ for $\det A \leq 0$, and 
\item $W$ is sufficiently smooth and the following consistency assumptions with respect to the differential operator $L$ hold true:
$W(\Id) = 0$, $\ DW(\Id) = 0$ and 
\begin{equation*}
\frac12 D^2W(\Id)(B,B)=\frac{\lambda}{2}(\tr B)^2+\mu\tr\!\left(\left(\frac{B+B^T}{2}\right)^2\right)\quad\forall B\in\R^{d,d}\,.
\end{equation*}
\end{list}
Furthermore, the set of admissible deformations is
\begin{equation*}
\admset = \{\phi\in W^{m,2}(D,D): \det D\phi > 0 \text{ a.e. in }D , \phi = \Id\text{ on }\partial D\}\,.
\end{equation*}
Note that we use the symbol $\Id$ both for the identity mapping $x\mapsto x$ and the identity matrix.
The first two assumptions ensure the existence of a minimizing deformation in \eqref{eq:WEnerDefinition} and thus the well-posedness of the discrete energy $\WEner[\image,\tilde \image]$ for $\image,\,\tilde \image \in \imagespace$. 
Note that \cite[Theorem 1]{Ba81} already implies the global invertibility (a.e.) of 
every $\phi\in\admset$ because $\admset\subset W^{1,p}(D)$ for a $p>d$.
The third assumption states that the definition of $\WEner$ is consistent with the underlying dissipation described by the quadratic form $L$. 

Now, we consider discrete curves $\imagevec=(\image_0,\ldots, \image_K) \in(\imagespace)^{K+1}$  in image space
and define a \emph{discrete path energy} as the sum of pairwise matching functionals $\WEner$ evaluated on consecutive images  of these discrete curves as follows 
\begin{equation}\label{eq:pathenergy}
\pathenergy_{K}[\imagevec]:=
 K \sum_{k=1}^{K}\WEner[\image_{k-1},\image_k]\,.
\end{equation}
We refer to \cite{RuWi12b} for the introduction of such a variational time discretization on shape manifolds.
Based on this path energy, we can define \emph{discrete geodesic paths} as follows.
\begin{definition}
Let $\image_A,\image_B\in\imagespace$ and $K\geq 1$. A \textit{discrete geodesic} connecting $\image_A$ and $\image_B$ is a discrete curve
in image space that minimizes  $\pathenergy_{K}$ over all discrete curves
$\imagevec=(\image_0,\ldots,\image_K)\in(\imagespace)^{K+1}$ with
$\image_0 = \image_A$ and $\image_K = \image_B$. 
\end{definition}

Due to the assumption (W3), the energy on the right-hand side of \eqref{eq:WEnerDefinition} scales quadratically in the displacement $\phi-\Id$, which itself is expected to scale linearly in the time step 
$\tau =\frac1K$. This already motivates the coefficient $K$ in front of the discrete path energy. For the rigorous justification, we refer to the proof of 
Theorem \ref{thm:gamma} on the $\Gamma$-convergence estimates.

In general, we want the energy density to fulfill two desirable properties: isotropy and rigid body motion invariance.
A suitable choice for an isotropic and rigid body motion invariant energy density $W$ in the case $d=2$, which fulfills the assumptions (W1-3), is given by 
\begin{equation}
W(D\phi)=
a_1\left(\tr(D\phi^TD\phi)\right)^q+
a_2(\det D\phi)^r
+a_3(\det D\phi)^{-s}
+a_4\
\label{eq:energyDensityExOne}
\end{equation}
with coefficients $a_1=\tfrac{2^{-q} \mu}{q}$,
$a_2=\tfrac{\lambda+\mu-\mu q-\mu s}{r^2+rs}$,
$a_3=\tfrac{\lambda+\mu-\mu q+\mu r}{rs+s^2}$ and
$a_4=\tfrac{\mu \left(q^2-rs-q (1+r-s)\right)-\lambda q}{q r s}$ and $q,r \geq 1$,
which is a special case of an Ogden material.  
Indeed, it is possible to choose for given $\lambda, \mu >0$ 
the parameters $q,\,r,\,s$ in such a way that the resulting coefficients  $a_1$, $a_2$ and $a_3$ are positive.
Obviously, $D\phi^T D \phi$ and $\det D \phi$ are invariant with respect to rotations of the 
observer frame. The third term of the energy density ensures the required response of the energy on strong compression.
Both rigid body motion invariance and also this compression response cannot be realized with a simple quadratic energy density.
For the definition of a corresponding energy density in the case $d=3$, we refer to \cite[Section 4.9/4.10]{Ci97}.

In the discrete path energy, two opposing effects can be observed. For a given discrete curve $\imagevec$ and (minimizing) deformations
$\phi_1,\ldots, \phi_K$, the last term penalizes intensity variations along the discrete motion path 
$(x, \phi_1(x), (\phi_2\circ \phi_1)(x), \ldots, (\phi_K \circ \ldots \circ \phi_1)(x))$, 
whereas the first two terms penalize deviations of the (discrete) flow along these discrete motion paths 
from rigid body motions. We will see that $K (\image_k \circ \phi_k-\image _{k-1})$ 
reflects a time discrete material derivative along the above discrete motion path, whereas the first two terms 
represent a discrete dissipation density.
Let us remark that minimizers of the discrete energy reversed in order are in general no minimizers of $\pathenergy_{K}$ 
for the reversed boundary constraint $\image_0 = \image_B$ and $\image_K = \image_A$. Only asymptotically in the limit for $K\to \infty$, we will obtain this symmetry based on our convergence theory below.

\section{Well-posedness of the discrete path energy and existence of discrete geodesics}
\label{sec:wellposedness}
In this section, we will show that for images $\image,\,\tilde \image \in \imagespace$ a minimizing deformation in the definition of 
$\WEner[\image,\,\tilde \image]$ exists, which renders the definition of the discrete path energy well-posed. Furthermore, we will prove existence of a
minimizing path $\imagevec$ of the discrete path energy $\pathenergy_{K}$ and thereby establish the existence of a discrete geodesic.
 
\begin{proposition}[Well-posedness of $\WEner$]\label{wellPosednessWEner}
Under the above assumptions (W1-2) and for $\image,\tilde\image\in\imagespace$, there
exists a deformation $\phi\in \admset$ depending on $\image$ and $\tilde\image$ such that 
$\WEner[\image,\tilde\image]=\W[\image,\tilde\image,\phi]$, where
\begin{equation*}
\W[\image,\tilde\image,\phi]:=\int_D W(D \phi)+\gamma |D^m \phi|^2+\frac1{\delta}|\tilde\image \circ \phi-\image|^2 \d x\,.
\end{equation*}
Moreover, $\phi$ is a diffeomorphism and $\phi^{-1}\in C^{1,\alpha}(\bar D)$ for $\alpha\in(0,m-1-\frac{d}{2})$.
\end{proposition}
\begin{proof}
The proof proceeds in four steps.

{\em Step 1}. Due to (W1), we know that 
$0 \leq \underline{\mathbf{W}} :=\inf_{\phi \in\admset}\W[\image,\tilde\image,\phi]$ and since
$\Id\in\admset$ we have that $\W[\image,\tilde\image,\Id]<\infty$. 
Consider a minimizing sequence $(\phi^j)_{j\in\N} \subset\admset$ with monotonously decreasing energy 
$\W[\image,\tilde\image,\phi^j]<\infty$ that converges to $\underline{\mathbf{W}}$. 
In particular, $\overline{\mathbf{W}} = \W[\image,\tilde\image,\phi^1] < \infty$ is an upper bound.
As a consequence of Korn's inequality and the Gagliardo-Nirenberg inequality for bounded domains (see \cite[Theorem 1]{Ni66}),
we can deduce that the minimizing sequence is bounded in $W^{m,2}(D)$. Hence, due to the reflexivity of this space, there is a weakly convergent subsequence in $W^{m,2}(D)$, again denoted by 
$\phi^j$, such that $\phi^j\rightharpoonup\phi$ and by the Sobolev embedding theorem we can assume uniform convergence of 
$\phi^j\rightarrow\phi$ in $C^{1,\alpha}(\bar D)$ for $\alpha\in(0,m-1-\frac{d}{2})$. 

{\em Step 2}. 
We show that the deformation $\phi$ belongs to $\admset$. To this end, we will control the measure of the set
$S_\epsilon =\setof{x\in D}{\det  D \phi \leq \epsilon}$ for sufficiently small $\epsilon>0$.
Indeed, by using (W1), (W2) and Fatou's lemma, we 
obtain 
\begin{align*}
&\beta_0 \epsilon^{-s}|S_\epsilon|
\leq \beta_0 \int_{S_\epsilon} (\det D\phi)^{-s} \d x
\leq \int_{S_\epsilon} W(D\phi) \d x+\beta_1|D|\\
&\leq  \liminf_{j\to\infty}\int_{S_\epsilon} W(D\phi^j) \d x+\beta_1|D|
\leq \overline{\mathbf{W}}+\beta_1|D|
\end{align*}
and thus  $|S_\epsilon|\leq \frac{(\overline{\mathbf{W}}+\beta_1|D|)\epsilon^s}{\beta_0}$, which shows $|S_0|=0$ and
$\det  D\phi>0$ a.~e. on $D$. This implies $\phi\in\admset$ (note $\phi\in W^{1,p}$ for a $p>d$) and 
due to \cite[Theorem 1]{Ba81} and $\phi\in W^{m,2}(D)$ the deformation $\phi$ is injective and a homeomorphism.
By Sard's theorem for H\"older spaces (cf. \cite{BoHaSt05}) we additionally know that $(\phi^j)^{-1}$,  $\phi^{-1}$ are uniformly bounded in  $C^{1,\alpha}(\bar D)$.

{\em Step 3}. Next, we consider the convergence of the matching functional. To this end, using the above diffeomorphism property, we estimate  
\beqa
&& \int_D |\tilde u \circ \phi^j- u |^2-|\tilde u \circ \phi-u|^2 \d x \leq \int_D (|\tilde u \circ \phi^j- u |+|\tilde u \circ \phi-u|) |\tilde u \circ \phi^j- \tilde u \circ \phi | \d x \\
&& \leq C \left(\|\tilde u \circ \phi^j\|_{L^2(D)} + \|\tilde u \circ \phi\|_{L^2(D)} + \|u\|_{L^2(D)}\right) \|\tilde u \circ \phi^j- \tilde u\circ \phi\|_{L^2(D)}\\
&& \leq C \left(\|\tilde u\|_{L^2(D)} +\|u\|_{L^2(D)}\right) \|\tilde u - \tilde u\circ \psi^j\|_{L^2(D)}
\eeqa
with $\psi^j = \phi\circ (\phi^{j})^{-1}$. Due to the convergence of $\psi^j$ to the identity in $C^{1,\alpha}(\bar D)$, we observe that the right hand side of the above estimate 
convergences to $0$. To see this, we can approximate $\tilde u$ in $L^2(D)$ by a sequence of $C^1$ functions $(\tilde u_i)_{i\in \N}$ and obtain 
$$ \|\tilde u - \tilde u\circ \psi^j\|_{L^2(D)} \leq \|\tilde u - \tilde u_i\|_{L^2(D)} + \|\tilde u_i - \tilde u_i\circ \psi^j\|_{L^2(D)} + \|\tilde u_i \circ \psi^j\ - \tilde u \circ \psi^j\|_{L^2(D)}\,.$$
The first and the third term on the right hand side converge to $0$ for $i\to \infty$ and fixed $j$, whereas the second term converges to zero for $j\to \infty$ and fixed $i$.
This establishes the convergence of the matching functional.

{\em Step 4}.  Finally, we show the lower semicontinuity for the whole functional. Let $j(\epsilon)\in\N$ be such that
$
\W[\image,\tilde\image,\phi^j] \leq \W[\image,\tilde\image,\phi^{j(\epsilon)}] \leq \underline{\mathbf{W}} + \epsilon
$
for all  $j \geq j(\epsilon)\,$. Furthermore, we can enlarge $j(\epsilon)$ if necessary such that for all $j\geq j(\epsilon)$
\beq
\left|\int_{D} |\tilde\image \circ \phi^j{(x)} - \image (x)|^2 -
|\tilde\image \circ \phi (x) - \image (x)|^2 \d x \right|\le \epsilon\,.
\eeq
Again using (W1), (W2) and Fatou's lemma, we infer
{\allowdisplaybreaks
\begin{align*}
&\W[\image,\tilde\image,\phi] = \int_D W(D\phi) + \gamma |D^m \phi|^2 + \frac{1}{\delta}|\tilde\image \circ \phi - \image|^2 \d x\\ 
&\leq  \liminf_{j\to\infty} \int_{D}  W (D\phi^j) +\gamma |D^m \phi^j|^2 +
\frac{1}{\delta}|\tilde\image \circ \phi^j - \image|^2 \d x+\frac{\epsilon}{\delta}
\leq  \underline{\mathbf{W}} + \epsilon + \frac{\epsilon}{\delta} \,,
\end{align*}
} 
which proves the claim.
\end{proof}

Next, for a given discrete path $\imagevec=(\image_0,\ldots, \image_K)\in(\imagespace)^{K+1}$,
we define a \textit{discrete path energy} explicitly depending on a $K$-tuple of 
deformations $\defvec=(\phi_1,\ldots, \phi_K)\in\admset^K$ as follows:
\begin{equation*}
\ESPhi{K}[\imagevec,\defvec]:= K \sum_{k=1}^{K}\W[\image_{k-1},\image_k,\phi_k]\,.
\end{equation*}
As an immediate consequence of Proposition \ref{wellPosednessWEner}, there exists a vector of deformations $\defvec \in\admset^K$
such that $\ESPhi{K}[\imagevec,\defvec]=\pathenergy_{K}[\imagevec]$.
If the images $\image_0,\ldots, \image_K$ are sufficiently smooth, the corresponding system of 
Euler-Lagrange equations for $\phi_k$ is given by
\begin{equation*} 
\int_D W_{,A}(D \phi_k) : D \theta + 2 \gamma D^m\phi_k : D^m \theta + \frac{2}{\delta} (\image_k \circ \phi_k-\image_{k-1}) (\nabla \image_k \circ \phi_k) \cdot \theta \, \d x =0
\end{equation*}
for all $1\leq k\leq K$ and all test deformations $\theta\in W^{m,2}(D,\R^d) \cap W^{1,2}_0(D,\R^d)$, which is a system of nonlinear PDEs of order $2m$. Here ``:'' denotes the sum over all pairwise products of two tensors.

Before we discuss the existence of discrete geodesics, we first present the following partial result, 
which can be regarded as a counterpart of
Proposition \ref{wellPosednessWEner} because it establishes the existence of an energy minimizing vector of images $\imagevec$ for a given vector of deformations $\defvec$.

\begin{proposition}\label{existenceImageInnerVec}
Let $\image_A,\image_B\in\imagespace$ and $K\geq 2$. Assume a vector $\defvec\in\admset^K$ is given.
Then, there exists a unique $\imagevec = (\image_0,\ldots, \image_K) \in(\imagespace)^{K+1}$ with $\image_0= \image_A$, $\image_K= \image_B$ 
such that
\begin{equation*}
\ESPhi{K}[\imagevec,\defvec]=\inf_{\tilde\imagevec\in(\imagespace)^{K+1},\, \tilde \image_0= \image_A,\, \tilde \image_K= \image_B}\ESPhi{K}[\tilde\imagevec,\defvec]\,.
\end{equation*}
\end{proposition}

\begin{proof}
Let $\imageinnervec^j=(\image_1^j,\ldots,\image_{K-1}^j)\subset(\imagespace)^{K-1}$ be a minimizing sequence for the energy
$\ESPhi{K}[(\image_A,\cdot,\image_B),\defvec]$. With $\overline{\ESPhi{K}}$ as a finite upper bound for this energy along this sequence.
This upper bound is obtained setting $u_k = \tfrac{k}{K} u_B+  (1-\tfrac{k}{K}) u_A$.
Thanks to the estimate 
\beqn
\|\image_k^j\|_{2}\leq \|\image_{k+1}^j\circ\phi_{k+1}-\image_k^j\|_{2}+\|\image_{k+1}^j\circ\phi_{k+1}\|_{2}
\leq \left(\delta\overline{\ESPhi{K}}\right)^{\frac12}K^{-\frac12}+\|\image_{k+1}^j\circ\phi_{k+1}\|_{2}
\label{eq:UniformBoundImages}
\eeqn
we can deduce via induction (starting from $k=K-1$)
that $\imageinnervec^j$ is uniformly bounded in $(L^2(D))^{K-1}$ independent of $j$.
Thus, there exists a weakly convergent subsequence in $(L^2(D))^{K-1}$ with weak limit $\imageinnervec=(\image_1,\ldots, \image_{K-1})$.

We still have to show the uniqueness of this minimizer.
To this end we take into account the transformation rule 
\beqa
&& \int_D (\image_{k} \circ \phi_{k} - \image_{k-1})^2 + (\image_{k+1} \circ \phi_{k+1} - \image_k)^2\d x \\
 && = \int_D (\image_{k}  - \image_{k-1} \circ \phi_{k}^{-1})^2  (\det D \phi_{k})^{-1} \circ \phi_k^{-1} + (\image_{k+1} \circ \phi_{k+1} - \image_k)^2\d x
\eeqa
and derive from the Euler-Lagrange equation $\partial_{\image_k} \ESPhi{K}[\imagevec,\defvec] =0$ the pointwise 
condition 
\begin{equation*}
\left((\image_k -\image_{k-1}\circ \phi_{k}^{-1}) \left((\det D \phi_{k})^{-1} \circ \phi_k^{-1}\right) +
(\image_{k}-\image_{k+1} \circ \phi_{k+1})\right)(x) = 0\text{ for a.e. }x \in D\,,
\end{equation*}
which can also be written as 
\beqn\label{eq:Linfty}
\image_{k}(x) = \frac{\image_{k+1}\circ \phi_{k+1}(x) +
(\image_{k-1}\circ \phi_{k}^{-1}(x)) ((\det D \phi_{k})^{-1} \circ \phi_{k}^{-1}(x)) }{1+(\det D \phi_{k})^{-1} \circ \phi_{k}^{-1}(x)}
\eeqn
for a.e. $x \in D$.
This leads to a linear system of equations for
$(\image_1,\ldots, \image_{K-1})$, where evaluations at deformed positions are combined with evaluations
at non-deformed positions, which we can consider as a block tridiagonal operator equation.
In fact, defining for each $x\in D$ the discrete transport path 
\[X(x) = ( X_0(x), X_1(x),X_2(x),\ldots, X_{K}(x))^T \in\R^{K+1}\]
with $X_0(x) = x$ and 
$X_k(x) = \phi_k(X_{k-1}(x))$ for $k\in\{1,\ldots, K\}$ and the vector of associated intensity values
\begin{equation}\label{eq:Xk}
U(\imageinnervec,\defvec)(x) := (\image_1(X_1(x)),\image_2(X_2(x)),\ldots, \image_{K-1}(X_{K-1}(x)))^T\in\R^{K-1}
\end{equation}
we obtain for $K\geq 3$ and a.e. $x\in D$ 
a linear system of equations 
\beqn \label{eq:linearSystem}
\Matrix[\defvec](x) U(\imageinnervec,\defvec)(x) = R[\defvec](x) 
\eeqn
on $\R^{K-1}$.
In this case, $\Matrix[\defvec](x) \in\R^{K-1,K-1}$ is a tridiagonal matrix with 
\begin{align*}
(\Matrix[\defvec](x))_{k,k+1} =& -\frac1{1+(\det D \phi_{k})^{-1}\circ \phi_{k}^{-1}(X_k(x))}=-\frac1{1+(\det D \phi_{k})^{-1}(X_{k-1}(x))}\,, \\
(\Matrix[\defvec](x))_{k,k} =& +1 \,,\\
(\Matrix[\defvec](x))_{k,k-1} =& -\frac{(\det D \phi_{k})^{-1} \circ \phi_{k}^{-1}(X_k(x))}
{1+(\det D \phi_{k})^{-1} \circ \phi_{k}^{-1}(X_k(x))}=-\frac{(\det D \phi_{k})^{-1}(X_{k-1}(x))}
{1+(\det D \phi_{k})^{-1}(X_{k-1}(x))}\,,
\end{align*}
and $R[\defvec](x)\in\R^{K-1}$ is given by
\begin{equation*}
R[\defvec](x) =
\left(\frac{\image_{A}(x)(\det D \phi_{1})^{-1}(x)}{1+(\det D \phi_{1})^{-1}(x)}\ ,\ 0\ ,\ \ldots\ ,\ 0\ ,\ \frac{\image_{B}(X_K(x))}
{1+(\det D \phi_{K-1})^{-1}(X_{K-2}(x))}\right)^T.
\end{equation*}
For any vector of regular deformations $\defvec\in\admset^K$, we recall that  $\det D\phi_k >0$ for $k=1,\ldots, K$ and 
$\defvec\in (C^1(D))^K$. From this we deduce that for a.e. $x\in D$ 
the matrix $\Matrix[\defvec](x)$ is irreducibly diagonally dominant, which implies invertibility. Thus, for all $x\in D$  there exists a unique solution $U(\imageinnervec,\defvec)(x)$ solving \eqref{eq:linearSystem}.
\end{proof}
\begin{remark}[Inherited regularity]\label{remarklinear}
(i) If the input images $u_A$ and $u_B$ are in $L^\infty(D)$, then the images $u_1,\ldots, u_{K-1} \in L^\infty(D)$ and they share the same upper and lower bound 
as the input images. This follows immediately from the fact that $u_k(X_k(x))$ can be written as a convex combination of $u_{k-1}(X_{k-1}(x))$ and $u_{k+1}(X_{k+1}(x))$ for $k=1,\ldots, K-1$ due to \eqref{eq:Linfty}.\\
(ii) If the input images $\image_A$ and $\image_B$ are in $C^{0,\alpha}(\bar D)$ for $\alpha \leq m-1-\frac{d}2$, 
then the proof of Theorem \ref{existenceImageInnerVec} also shows that $\image_k \in C^{0,\alpha}(\bar D)$ for all $k=1,\ldots, K-1$.\\
(iii) The intensity values along the discrete transport path $X(x)$ depend in a unique way on
the values at the two end points $x$ and $X_{K}(x)$ and each $\image_k(X_k(x))$ is a weighted average of the
intensities $\image_A(x)$ and $\image_B(X_K(x))$, where the weights reflect the compression and expansion associated with 
the deformations along the discrete transport paths.
\end{remark}

Now, we are in the position to prove the existence of discrete geodesics making use of the existence of a minimizing family of deformations  
for the energy $\ESPhi{K}$ and a given discrete image path as a consequence of Proposition \ref{wellPosednessWEner} and 
the existence of an optimal discrete image path for a given family of deformations as stated in Proposition \ref{existenceImageInnerVec}.
\begin{theorem}[Existence of discrete geodesics]
Let $\image_A,\image_B\in\imagespace$ and $K\geq 2$. Then there exists $\imageinnervec\in(\imagespace)^{K-1}$ such that
$$\pathenergy_{K}[(\image_A,\imageinnervec,\image_B)]=\inf_{\imageinnervectest\in(\imagespace)^{K-1}}\pathenergy_{K}[(\image_A,\imageinnervectest,\image_B)]\,.$$
\end{theorem}
\begin{proof} Let us assume that $(\imageinnervec^j)_{j\in\N} \in(\imagespace)^{K-1}$ with $\imageinnervec^j = (\image_1^j, \ldots, u^j_{K-1})$ 
is a minimizing sequence of the discrete 
path energy $\pathenergy_{K}[(\image_A,\cdot,\image_B)]$, where $\overline{\pathenergy}_{K}$ is an upper bound of the discrete path energy.
Due to Proposition \ref{wellPosednessWEner}, for every $\imageinnervec^j$ there exists 
a family of optimal deformations 
$\defvec^j = (\phi^j_1,\ldots, \phi^j_K) \in \admset^K$ with 
$\ESPhi{K}[(\image_A,\imageinnervec^j,\image_B),\defvec^j]\leq \ESPhi{K}[(\image_A,\imageinnervec^j,\image_B), \mathbf{\Phi}']$ 
for all $\mathbf{\Phi}' \in \admset^K$.
Furthermore,  we can assume (by possibly replacing $\imageinnervec^j$ and thereby further reducing the energy) that $\imageinnervec^j$ already minimizes the 
discrete path energy $\ESPhi{K}[(\image_A,\imageinnervectest,\image_B),\defvec^j]$ over all $\imageinnervectest \in (\imagespace)^{K-1}$.
We note that due to the coercivity estimate  
$\|D^m \phi_k^j\|_2^2\leq\frac{\overline{\pathenergy}_{K}}{\gamma}$ and the Gagliardo-Nirenberg inequality
the deformations $\phi^j_k$ are uniformly bounded in $W^{m,2}(D,\R^d)$ for $k=1,\ldots, K$.
Together with the compact embedding of $W^{m,2}(D,\R^d)$ into $C^{1,\alpha}(\bar D,\R^d)$ for $0<\alpha<m-1-\frac{d}{2}$,
this implies that (up to the selection of another subsequence) 
$\defvec^j$ converges to $\defvec=(\phi_1,\ldots,\phi_{K})$ 
weakly in $(W^{m,2}(D,\R^d))^K$ and uniformly in $(C^{1,\alpha}(\bar D,\R^d))^K$.
Following the same line of arguments as in Step 2 of the proof of Proposition \ref{wellPosednessWEner},
we in addition infer that $\det D\phi_k>0$ a.e. in $D$ for $k=1,\ldots, K$ and thus
$\defvec\in\admset^K$. 

Due to \eqref{eq:UniformBoundImages} we know that 
the resulting images $\image^j_k$, which are associated with the above subsequence of deformations, are uniformly bounded for $k=1,\ldots, {K-1}$
in $\imagespace)$. Hence, a subsequence of $(\image^j_k)_{j\in \N}$ converges weakly in $L^2(D)$ to some $\image_k$.
Finally, we deduce from the strong convergence of $\defvec^j$ in $(C^{1,\alpha}(\bar D,\R^d))^K$ 
that 
\beq
\sum_{k=1}^K\int_D (\image_{k} \circ \phi_{k} - \image_{k-1})^2 \d x \leq
\liminf_{j \to \infty} \sum_{k=1}^K \int_D (\image^j_{k} \circ \phi^j_{k} - \image^j_{k-1})^2 \d x\,.
\eeq
Together with the weak lower semi-continuity of $\phi\mapsto \int_D W(D \phi)+\gamma |D^m \phi|^2 \d x$
we obtain with $\imageinnervec=(\image_1,\ldots, \image_{K-1})$  
that
\begin{align*}
&\pathenergy_{K}[\image_A, \imageinnervec, \image_B] = \ESPhi{K}[(\image_A, \imageinnervec,\image_B),\defvec] \\
\leq&  \liminf_{j\to \infty} \ESPhi{K}[(\image_A,\imageinnervec^j,\image_B), \defvec^j]
= \liminf_{j\to \infty} \pathenergy_{K}[\image_A,  \imageinnervec^j, \image_B]\,.
\end{align*}
This proves the claim.
\end{proof}

\section{Convergence of discrete geodesic paths}\label{section:Gamma}
In what follows, we will study the convergence of minimizers of our discrete variational model 
\eqref{eq:pathenergy} for $K\to \infty$ to minimizers of the 
continuous model \eqref{eq:DefinitionPathenergy} and thus the convergence of discrete geodesic paths to continuous geodesic paths.
To this end, we prove $\Gamma$-convergence estimates for a natural extension of the discrete path energy. 
For an introduction to $\Gamma$-convergence, we refer to \cite{Da93}.

At first, let us discuss a suitable interpolation of continuous paths.
 For fixed $K\geq 2$ and time step size $\tau = \frac1K$, let $t_k = k\tau$
denote the time step corresponding to a vector of images 
$\imagevec=(\image_0,\ldots,\image_K) \in (\imagespace)^{K+1}$.  For a vector  
$\defvec=(\phi_1,\ldots,\phi_K)\in\admset^K$ of optimal deformations resulting from the 
minimization in \eqref{eq:WEnerDefinition}, we define for 
$k=1,\ldots, K$ the motion field $v_k = K (\phi_{k}-\Id)$ and the induced transport map
$y_k(t,x) = x + (t-t_{k-1}) v_k(x)$ with $t\in [t_{k-1},t_k]$. Note that
$y_k(t_{k-1},x)=x$ and $y_k(t_k,x) = \phi_k(x)$.
If one assumes that $\funcnorm{D\phi_k-\Id}{\infty} := \sup_{x\in D} \max_{|v|=1}|(D\phi(x)-\Id)v|<1$, then 
$y_k(t,\cdot) = \Id + K (t-t_{k-1}) (\phi_k-\Id)$ is invertible. Thus, denoting the inverse of $y_k(t,\cdot)$
by $x_k(t,\cdot)$ one obtains the image interpolation $\imagecurve = \interpolimage_{K}[\imagevec,\defvec]$ with 
\begin{equation}
\interpolimage_{K}[\imagevec,\defvec](t,x)=\image_{k-1}(x_k(t,x))+K(t-t_{k-1})(\image_{k}\circ\phi_k-\image_{k-1})(x_k(t,x))
\label{eq:ImagecurveRep}
\end{equation}
for $t\in [t_{k-1},t_k]$. This interpolation represents on each interval $[t_{k-1},t_k]$ the 
blending between the images $\image_{k-1}=
\interpolimage_{K}[\imagevec,\defvec](t_{k-1},\cdot)$ and 
$\image_k=\interpolimage_{K}[\imagevec,\defvec](t_{k},\cdot)$ along affine transport paths 
$$\setof{(t,y_k(t,x))}{t\in [t_{k-1},t_k]}$$ for $x\in D$. 
Based on this interpolation, a straightforward extension $\continuouspathenergy_{K}:L^2((0,1)\times D)\rightarrow[0,\infty]$ 
of the discrete path energy $\pathenergy_{K}$ is given by 
\beq
\continuouspathenergy_{K}[\imagecurve]=
\left\{
\begin{array}{cl}
\pathenergy^D_{K}[\imagevec,\defvec] & ;\; \text{if }\;\imagecurve = \interpolimage_{K}[\imagevec,\defvec]\;
\mbox{with }\; \imagevec \in (\imagespace)^{K+1} \mbox{ and } \\
& \;\; \defvec \mbox{ is a minimizer of } \pathenergy^D_{K}[\imagevec,\cdot] \mbox{ over } \admset^K \\
+ \infty & ;\; \text{else}
\end{array}
\right. \,.
\eeq
Now, we are in the position to discuss the $\Gamma$-convergence estimates. 
The statements of the theorem are sufficient to prove that subsequences of discrete geodesics converge to a continuous geodesic (cf. Theorem \ref{thm:convergence}).
\begin{theorem}[$\Gamma$-convergence estimates]\label{thm:gamma}
Under the assumptions (W1-3), the time discrete path energy $\continuouspathenergy_K$ $\Gamma$-converges to the time continuous path energy $\continuouspathenergy$ in the following sense.
The estimate $\liminf_{K\to\infty} \continuouspathenergy_K[\imagecurve^K]\geq 
\continuouspathenergy[\imagecurve]$ holds for every sequence $(\imagecurve^K)_{K\in \N}\subset L^2((0,1)\times D)$ with $\imagecurve^K\rightharpoonup \imagecurve$ (weakly) in 
$L^2((0,1)\times D)$. Furthermore, for $u\in L^2((0,1)\times D)$
there exists a sequence $(\imagecurve^K)_{K\in \N}\subset L^2((0,1)\times D)$ with $\imagecurve^K\rightarrow \imagecurve$ in 
$L^2((0,1)\times D)$ such that the estimate $\limsup_{K\to\infty} \continuouspathenergy_K[\imagecurve^K]\leq \continuouspathenergy[\imagecurve]$ holds.
\end{theorem}
\medskip

Let us at first briefly outline the structure of the proof to facilitate the reading.
The proof itself refers to the outline with corresponding paragraph headlines.
To verify the $\liminf$ estimate we proceed as follows:
\begin{itemize}
\item[(i)] {\it Reconstruction of a flow and a weak material derivative.} 
For a sequence of images $\imagecurve^{K} = \interpolimage_K[\imagevec^K,\defvec^K]$ in $L^2((0,1)\times D)$ with 
$\imagevec^K=(\image_0^K,\ldots,\image_K^K) \in (\imagespace)^{K+1}$, we consider a set of associated optimal matching deformations and construct the induced 
underlying motion field, for which the mismatch energy turns out to be the weak material derivative in the limit.
\item[(ii)] {\it Weak lower semicontinuity of the path energy.}
Using a priori bounds for the sequence of motion fields and material derivatives, we obtain weakly convergent subsequences and using a Taylor expansion of 
the energy density function $W$, we show a lower semicontinuity result required for the $\liminf$ inequality.
\item[(iii)] {\it Identification of the limit of the material derivatives as the material derivative for the limit image sequence.}
We still have to show that the pair of the weak limits of the velocity fields and the material derivatives is indeed an instance of a tangent vector at the limit image.
The core insight is that instead of taking the limit in the defining equation \eqref{eq:weakmaterial1} of the weak material derivative in Eulerian coordinates 
one has to use an equivalent flow formulation in Lagrangian coordinates.
\item[(iv)] {\it Convergence of the discrete image sequences pointwise everywhere in time.} In step (iii) we need that an image sequence with bounded path energy  
converges not only weakly in $L^2((0,1)\times D)$, but for every time $t\in [0,1]$ the image sequence evaluated at that time converges already weakly in $L^2(D)$. 
We use a trace theorem type argument to verify this.
\end{itemize}
The proof of the $\limsup$ estimate consists of the following steps:
\begin{itemize}
\item[(i)] {\it Construction of the recovery sequence.}  The key observation is that the construction of a recovery sequence is not based on some (time-averaged) interpolation of the given image path $\imagecurve \in L^2((0,1)\times D)$. In fact, one considers for fixed $K$ a local time averaging of an underlying motion field leading to a bounded path energy,
and constructs from this via integration of the associated material derivative along the induced transport path a discrete family of images $(u^K_0,\cdots, u^K_K)$.
\item[(ii)] {\it Proof of the the $\limsup$ inequality.} The key ingredient for the proof of the $\limsup$ inequality is the convexity of the total viscous dissipative functional, 
which we exploit based on the above construction of the recovery sequence via an application of Jensen's inequality. This requires that the discrete motion fields are 
indeed defined via local time averaging of the given continuous motion field. Furthermore, we again use a Taylor expansion of the energy density function $W$.
\item[(iii)] {\it Convergence of the discrete image sequences.}
Due to the fact that the recovery sequence of images $(\imagecurve ^K)_{K\in \N}$ is defined via integration of the material derivative and not by simple time averaging, 
we are still left to verify that $\imagecurve^K$ converges to $\imagecurve$ in  $L^2((0,1)\times D)$. 
\end{itemize}
\medskip
 
\clearpage

\begin{proof}
Throughout the proof we will use a generic constant $C$ independent of $K$.\medskip

{\it The liminf---estimate:}\medskip

{\it (i) Reconstruction of a flow and a weak material derivative.}
Let $\{\imagecurve^{K}\}_{K\in\N}\subset L^2((0,1)\times D)$ be any sequence of images
that converges weakly in $L^2((0,1)\times D)$ to $\imagecurve\in L^2((0,1)\times D)$. To exclude trivial
cases, \ie $\liminf_{K\to\infty} \continuouspathenergy_K[\imagecurve^K]=\infty$, we may assume $\continuouspathenergy_K[\imagecurve^{K}]\leq \overline\continuouspathenergy<\infty$ for all $K\in\N$, which implies $\imagecurve^{K} = \interpolimage_K[\imagevec^K,\defvec^K]$ for $\imagevec^K=(\image_0^K,\ldots,\image_K^K) \in (\imagespace)^{K+1}$ and
an associated vector of deformations 
$\defvec^K=(\phi^K_1, \ldots, \phi^K_K)$, which is defined as a vector of (not necessarily unique) solutions of the pairwise matching problems \eqref{eq:WEnerDefinition}. 
Each $\defvec^K$ generates on each time interval $[t_{k-1},t_k)$ affine transport paths with motion velocity $\tilde v^K_k(t,y)=K(\phi^K_k-\Id)(x^K_k(t,y))$. Here, we use the notation $t_k=\tfrac{k}{K}$ (for the sake of brevity without explicit reference to the  sequence index $K$) and $x^K_k$ is the above defined pullback associated with the deformation $\phi^K_k$ on the interval $[t_{k-1},t_{k}]$.
As it will be shown below in \eqref{eq:uniformBoundPhi}, for sufficiently large $K$ a piecewise affine reconstruction of 
$\imagecurve^K$ along straight line segments from $x$ to $\phi^K_k(x)$  can be performed using \eqref{eq:ImagecurveRep}. Thus, the difference quotient $K \left(\image^K_{k}(\phi^K_{k}(x))-\image^K_{k-1}(x)\right)$ is the material derivative of $\imagecurve^K$ for all $y^k(t,x)$ with $t\in (t_{k-1},t_k)$, \ie
\begin{equation}
z^K(t,y)=\frac{\d}{\d s} \imagecurve^{K}(t+s,y+s \tilde v^{K}_k(t,y))\big|_{s=0}=K 
\left(\image^K_{k}\circ\phi^K_{k}-\image^K_{k-1}\right)(x^K_k(t,y))
\label{eq:material}
\end{equation}
is the classical material derivative of $\imagecurve^K$.
Hence, the regularity of $\defvec^K$ stated in Proposition \ref{wellPosednessWEner} implies that  $z^K$ fulfills the equation for the weak material derivative \eqref{eq:definitionMaterialDerivative}, \ie
\beqn\label{eq:weakmaterial}
\int_D \int_0^1 z^K\vartheta\d t \d x
=-\int_D \int_0^1 (\partial_t\vartheta + \operatorname{div}(\tilde v^K\vartheta))\imagecurve^K \d t \d x
\eeqn
for all $\vartheta\in W^{1,2}_0((0,1)\times D)$ and with $\tilde v^K(t,y)=\tilde v^K_k(t,y)$ for $t\in [t_{k-1},t_{k})$.
Let us remark that $\tilde v^K(t,\cdot)$ vanishes on the boundary $\partial D$ for $t\in(0,1)$, which corresponds to the assumption on the continuous velocity $v$ in the metamorphosis model from the introduction.
As a next step we show
\beqn
\lim_{K\rightarrow\infty}\int_D \int_0^1 \left|z^K\right|^2\d t\d x=
\lim_{K\rightarrow\infty}K\sum_{k=1}^K \int_D |\image^K_{k}\circ\phi^K_{k}-\image^K_{k-1}|^2 \d x\,.
\label{ReexpressSquareMaterialDerivative}
\eeqn
Indeed, using \eqref{eq:material} one obtains 
\begin{align*}
&\int_D \int_{t_{k-1}}^{t_{k}}\left|z^K\right|^2\d t\d x=
\int_D \int_{t_{k-1}}^{t_{k}} K^2\left(\left(\image^K_{k}\circ\phi^K_{k}-\image^K_{k-1}\right)(x^K_k(t,x))\right)^2\d t\d x\\
=&\int_D \int_{t_{k-1}}^{t_{k}} K^2\left(\left(\image^K_{k}\circ\phi_{k}-\image^K_{k-1}\right)(x)\right)^2
\det D y_k^K(t,x)\d t\d x\,,
\end{align*}
where $D y_k^K(t,x) = \Id+K(t-t_{k-1})(D\phi^K_k(x)-\Id)$. From the uniform bound on the energy, we deduce
\begin{equation}
\sum_{k=1}^K \int_D K(\image^K_k\circ\phi^K_k-\image^K_{k-1})^2\d x \leq \delta\overline\continuouspathenergy\,.\label{eq:defestimate}
\end{equation}
Furthermore, we can estimate 
\begin{equation*} 
\funcnorm{\det\!\left(\Id+K(\cdot \!-\!t_{k-1})(D\phi^K_k\!-\!\Id)\right)\!-\!1}{L^{\infty}((\frac{k-1}{K},\frac{k}{K})\times D)} \leq{} C \|\phi^K_k-\Id\|_{C^1(\bar D)}\,.
\end{equation*}
The Sobolev estimate $\funcnorm{\phi-\Id}{C^{1,\alpha}(\bar D)} \leq C \funcnorm{\phi-\Id}{W^{m,2}(D)}$ for $\alpha \leq 
m-1-\tfrac{d}{2}$ and the Gagliardo-Nirenberg interpolation inequality $\funcnorm{\phi-\Id}{W^{m,2}(D)}\leq C\funcnorm{D^m\phi}{L^2(D)}$ (cf. \cite{Ni66}) for $\phi \in W^{m,2}(D)\cap W^{1,2}_0(D)$ imply
\begin{equation} \label{eq:uniformBoundPhi}
\funcnorm{\phi^K_k-\Id}{C^{1,\alpha}(\bar D)}^2
\leq \sum_{l=1}^K C\funcnorm{D^m \phi^K_l}{L^{2}(D)}^2
\leq\frac{C\overline\continuouspathenergy}{\gamma K}\,.
\end{equation}
Together with \eqref{eq:defestimate} this proves \eqref{ReexpressSquareMaterialDerivative}.

{\it(ii)  Weak lower semicontinuity of the path energy.}
Next, from  \eqref{eq:defestimate} and \eqref{ReexpressSquareMaterialDerivative} we deduce that the material derivatives $z^K$ are uniformly bounded in $L^2((0,1)\times D)$ independent of $K$.
Thus, there exists a subsequence, again denoted by $(z^K)_{K\in \N}$, which converges weakly in $L^2((0,1)\times D)$ to some $z\in L^2((0,1)\times D)$ as 
$K\rightarrow\infty$. By the lower semicontinuity of the $L^2$-norm, one achieves 
\beq
\int_D \int_0^1 \left|z\right|^2\d t\d x\leq
\liminf_{K\rightarrow\infty}\int_D \int_0^1 \left|z^K\right|^2\d t\d x\,.
\eeq
Now, we will prove that there exists a velocity field $v\in L^2((0,1),W^{1,2}_0(D)\cap W^{m,2}(D))$ such that $\overline{(v,z)}\in T_{\imagecurve}L^2$ and
\beq
\int_0^1\int_D  L[v,v] \d x\d t \leq
\liminf_{K\rightarrow\infty}
K\sum_{k=1}^K
\int_D W(D \phi_k^K) + \gamma |D^m \phi_k^K|^2\d x\,.
\eeq
The second order Taylor expansion around $t_{k-1}$ of the function 
$t\mapsto W(\Id+(t-t_{k-1})Dv_k^K)$ at  $t=t_k$ gives
\begin{align*}
W(D \phi^K_k)=&W(\Id)+\frac{1}{K}DW(\Id)(Dv^K_k)+\frac{1}{2K^2}D^{2}W(\Id)(Dv^K_k,Dv^K_k)+O(K^{-3}|Dv^K_k|^3)\\
=&\frac{1}{K^2}\!\left(\frac{\lambda}{2}\left(\tr\varepsilon[v^K_k]\right)^2+\mu\tr(\varepsilon[v^K_k]^2)\right)+O(K^{-3}|Dv^K_k|^3)
\end{align*}
with $v^K_k(x) = K (\phi^K_k(x)-x)$. The second equality follows from (W3).
Then 
\begin{eqnarray*}
&& K\sum_{k=1}^K \int_D W(D \phi_k^K) + \gamma |D^m \phi_k^K|^2\d x \\
&& \leq \frac1K \sum_{k=1}^K \int_D \frac{\lambda}{2}(\tr \varepsilon[v^K_k])^2+\mu\tr(\varepsilon[v^K_k]^2)+\gamma\left|D^m v^K_k\right|^2\d x
+ C \sum_{k=1}^K K\int_D \!\! K^{-3}|Dv_k^K|^3 \d x.
\end{eqnarray*}
The last term is of order $K^{-\frac{1}{2}}$, which follows from the boundedness of the energy and by applying \eqref{eq:uniformBoundPhi}, \ie
\begin{align*}
\sum_{k=1}^K \!K\!\int_D \! K^{-3}|Dv^K_k|^3 \d x
\leq C \! \max_{k=1,\ldots,K}\|\phi_k^K-\Id\|_{C^1(\bar{D})}
\sum_{k=1}^K K\!\funcnorm{\phi_k^K\!-\!\Id}{W^{m,2}(D)}^2
\leq C K^{-\frac{1}{2}}.
\label{eq:miscOrderEstimate}
\end{align*}
Next, for $K\to \infty$ the limes inferior of the remainder can be estimated as follows.
We define $v^K \in L^2((0,1)\times D)$ via $v^K(t,\cdot) = v^K_k$ for $t\in [t_{k-1},t_k)$.
Due to the uniform bound of the discrete path energy $v^K$ is uniformly bounded in $L^2((0,1),W^{m,2}(D))$ and up to the selection of a subsequence $v^K$
converges weakly in $L^2((0,1),W^{m,2}(D,\R^d) \cap W^{1,2}_{0}(D,\R^d))$ to some $v\in L^2((0,1),W^{m,2}(D,\R^d) \cap W^{1,2}_{0}(D,\R^d))$ for $K\to \infty$.
Then, by a standard weak lower semicontinuity argument we obtain
\beqa
&&\liminf_{K\to \infty} \frac1K \sum_{k=1}^K \int_D \frac{\lambda}{2}(\tr \varepsilon[v_k^K])^2+\mu\tr(\varepsilon[v_k^K]^2)+\gamma\left|D^m v_k^K\right|^2 \d x \\
&& = \liminf_{K\to \infty} \int_0^1 \int_D \frac{\lambda}{2}(\tr \varepsilon[v^K])^2+\mu\tr(\varepsilon[v^K]^2)+\gamma\left|D^m v^K\right|^2 \d x \d t\\
&& \geq \int_0^1 \int_D \frac{\lambda}{2}(\tr \varepsilon[v])^2+\mu\tr(\varepsilon[v]^2)+\gamma\left|D^m  v\right|^2 \d x \d t\,.
\eeqa

{\it (iii) Identification of the limit of the material derivatives as the material derivative for the limit image sequence.}
It remains to verify that we can pass to the limit in \eqref{eq:weakmaterial} for $K\to \infty$ with $v$ also being  the weak limit of $\tilde v^K$ in $L^2((0,1)\times D)$. This will indeed imply that 
$z$ is the weak material derivative for the image path $\imagecurve$ and the velocity field $v$ fulfilling \eqref{eq:weakmaterial1}
and hence $\overline{(v,z)} \in T_\imagecurve L^2(D)$.
To this end, the main difficulty is to prove the weak continuity of $(u,v) \mapsto u\div(v\eta)$. 
In  \cite[Theorem 2]{TrYo05a} (with the essential ingredient, which we actually required here, given in \cite[Lemma 6]{TrYo05a}) it is shown that 
for the family of diffeomorphisms $\psi:[0,1] \to C^{1}(\bar D)$  
resulting from the transport 
\beqn
\dot \psi(t,\cdot) = v(t,\psi(t,\cdot)) \label{eq:ODE}
\eeqn
for some velocity field $v\in L^2((0,1),W^{m,2}(D) \cap W^{1,2}_0(D))$ and for given 
initial data $\psi(0) = \Id$ the integral formula \beqn\label{eq:Int}
\imagecurve(t,x) = \imagecurve(0,\psi_{t,0}(x)) + \int_0^t z(s,\psi_{t,s}(x)) \d s 
\eeqn
for an image path $u$,  a function $z \in L^2((0,1),L^2(D))$ 
and for a.\,e. $x\in D$ with $\psi_{t,s} = \psi(s,(\psi(t,\cdot))^{-1})$ is equivalent to \eqref{eq:weakmaterial1}. 
We refer to \cite[Lemma 2.2]{DuGrMi98} for the existence of a unique solution $\psi$ of \eqref{eq:ODE}.
From \eqref{eq:material} we deduce that $(\imagecurve^K,\tilde v^K, z^K)$ obeys 
\beqn\label{eq:discreteInt}
\imagecurve^K(t,x) = \imagecurve^K(0,\psi^K_{t,0}(x)) + \int_0^t z^K(s,\psi^K_{t,s}(x)) \d s\,,
\eeqn
where $\psi^K_{t,s} = \psi^K(s,(\psi^K)^{-1}(t,\cdot))$ with $\psi^K:[0,1] \to C^{1}(\bar D)$ denoting the time discrete family of diffeomorphisms induced by the motion field $\tilde v^K$ 
and solving  
\beqn
\dot \psi^K(t,x) = \tilde v^K(t, \psi^K(t,x)) \label{eq:discreteODE}
\eeqn
for all $x\in D$. 
In what follows, we will show strong convergence of $\psi^K$ to $\psi$, for which \eqref{eq:ODE} holds.
At first, we observe that $\funcnorm{y^K_k(t,\cdot)}{C^{1,\alpha}(\bar D)}\leq C (1+K^{-1} \funcnorm{v^K_k(t,\cdot)}{C^{1,\alpha}(\bar D)})$ for  
$y^K_k(t,x) = x + (t-t_k) v^K_k(x)$ and $t\in [t_{k-1},t_k)$.
By Sard's theorem in H\"older spaces \cite{BoHaSt05} and \eqref{eq:uniformBoundPhi} we deduce that $\funcnorm{x^K_k(t,\cdot)}{C^{1,\alpha}(\bar D)}\leq C (1+K^{-1} \funcnorm{v^K_k(t,\cdot)}{C^{1,\alpha}(\bar D)})$ 
for the inverse $x^K_k(t,\cdot) = y^K_k(t,\cdot)^{-1}$. 
Using the definition of $\tilde v^K_k$, the $C^{1,\alpha}$-estimate for the concatenation of $C^{1,\alpha}$-functions, and \eqref{eq:uniformBoundPhi} we get 
\beqa
\funcnorm{\tilde v^K_k(t,\cdot)}{C^{1,\alpha}(\bar D)} &\leq& C \funcnorm{v^K_k(t,\cdot)}{C^{1,\alpha}(\bar D)} 
\left(1+K^{-1} \funcnorm{v^K_k(t,\cdot)}{C^{1,\alpha}(\bar D)}\right)\\ & \leq& C  \funcnorm{v^K_k(t,\cdot)}{C^{1,\alpha}(\bar D)}\,.
\eeqa 
The uniform boundedness of $v^K$ in $L^2((0,1),W^{m,2}(D))$ and the continuity of the embedding of $W^{m,2}(D)$ into $C^{1,\alpha}(\bar D)$ imply 
that $\tilde v^K$ is uniformly bounded in $L^2((0,1),C^{1,\alpha}(\bar D))$.  Following \cite[Lemma 7]{TrYo05a} (in a straightforward generalization for velocities uniformly bounded in
$L^1((0,1),C^{1,\alpha}(\bar D))$) one shows via  Gronwall's inequality that $\psi^K$ defined in \eqref{eq:discreteODE} is uniformly bounded in $L^\infty((0,1),C^{1,\alpha}(\bar D))$.
Finally, using this bound and once again the $C^{1,\alpha}$-estimate for the concatenation of $C^{1,\alpha}$-functions we obtain from \eqref{eq:discreteODE} the estimate 
\beqa
\funcnorm{\psi^K(t,\cdot)-\psi^K(s,\cdot)}{C^{1,\alpha}(\bar D)} &\leq& C \int_s^t \funcnorm{v^K(r,\cdot)}{C^{1,\alpha}(\bar D)} \d r \\
&\leq& C (t-s)^{\frac12} \left(\int_s^t  \funcnorm{v^K(r,\cdot)}{C^{1,\alpha}(\bar D)}^2 \d r\right)^{\frac12} 
\leq C (t-s)^{\frac12}\,,
\eeqa
which proves that $\psi^K$ is uniformly bounded in $C^{0,\frac12}([0,1], C^{1,\alpha}(\bar D))$. Thus, for  some $\beta$ with $0<\beta< \min \{\frac12, \alpha\}$ and up to the selection of a subsequence $\psi^K$ converges strongly  in $C^{0,\beta}([0,1], C^{1,\beta}(\bar D))$ to some $\psi \in C^{0,\frac12}([0,1], C^{1,\alpha}(\bar D))$ and $\psi$ solves \eqref{eq:ODE} (cf. \cite[Theorem 9]{TrYo05a}). 
The mapping $\left(t\mapsto(\psi^K(t,\cdot))^{-1}\right)_{K\in \N}$, which solves \eqref{eq:discreteODE} backward in time, is uniformly bounded in  $C^{0,\frac12}([0,1], C^{1,\alpha}(\bar D))$
(cf.  \cite[Lemma 9]{TrYo05a}).
Next, we obtain from \eqref{eq:discreteInt} for functions $u^K$ with bounded energy $\continuouspathenergy_{K}$ the following estimate:
\beqan
&&\funcnorm{u^K(t+\tau, \psi^K(t+\tau,\cdot))-u^K(t, \psi^K(t,\cdot))}{L^2(D)}^2 \nonumber \\
&&  \leq \int_D \left( \int_t^{t+\tau}  z^K(s,\psi^K(s,x))  \d s\right)^2 \d x \nonumber \\
&& \leq \tau \funcnorm{\det D ((\psi^K)^{-1})}{L^\infty((0,1)\times D)} 
\int_t^{t+\tau}\!\!\!\!\! \funcnorm{z^K(s,\cdot)}{L^2(D)}^2   \d s \nonumber \\
&& \leq C \tau  \funcnorm{z^K}{L^2((0,1)\times D)}^2 \leq C \tau 
\label{eq:TimeBoundImages}
\eeqan
for all $t \geq 0$, $\tau >0$ with $t+\tau \leq 1$. The analogous estimate holds for $u^K$, $\psi^K$, and $z^K$ replaced by $u$, $\psi$, and $z$, respectively (cf. \cite{TrYo05a}).
From this and the uniform smoothness of $\psi^K$ and $\psi$ we deduce that  
for a subsequence (again denoted by $(u^K)_{K\in \N}$)  $u^K(t)\rightharpoonup u(t)$ weakly in $L^2(D)$ for all $t\in [0,1]$.
A detailed verification is given in the last step of the proof below.
Then, multiplying \eqref{eq:discreteInt} with a test function $\eta \in C^{\infty}_c(D)$ and integrating over $D$ yields
\beqan
0 &=& \int_D \imagecurve^K(t,x) \eta(x) \d x - \int_D \imagecurve^K(0,\psi^K_{t,0}(x)) \eta(x) \d x -
\int_0^t \int_D z^K(s, \psi^K_{t,s}(x)) \eta(x) \d x  d s \nonumber\\
&=& \int_D \imagecurve^K(t,x) \eta(x) \d x - \int_D \imagecurve^K(0,y) \eta((\psi^K_{t,0})^{-1}(y)) (\det D \psi^K_{t,0})^{-1} ((\psi_{t,0}^K)^{-1}(y)) \d y \nonumber \\
&& - \int_0^t \int_D z^K(s, y) \eta((\psi^K_{t,s})^{-1}(y)) (\det D \psi^K_{t,s})^{-1} ((\psi_{t,s}^K)^{-1}(y))\d y \d s\,. \label{eq:materialdiscrete}
\eeqan
Based on the weak convergence of $u^K,\, z^K$ and the strong convergence of $t\mapsto (\psi^K(t,\cdot))^{-1}$ and 
$\psi^K$ we can pass to the limit in \eqref{eq:materialdiscrete} and obtain
\beqa
0&=&\int_D \imagecurve(t,x) \eta(x) \d x - \int_D \imagecurve(0,y)\, \eta((\psi_{t,0})^{-1}(y)) \,(\det D \psi_{t,0})^{-1}((\psi_{t,0})^{-1}(y)) \d y \\
&& - \int_0^t \int_D z(s, y)\, \eta((\psi_{t,s})^{-1}(y))\, (\det D \psi_{t,s})^{-1}((\psi_{t,s})^{-1}(y)) \d y \d s\\
&=& \int_D \imagecurve(t,x) \eta(x) \d x - \int_D \imagecurve(0,\psi_{t,0}(x)) \eta(x) \d x -
\int_0^t \int_D z(s, \psi_{t,s}(x)) \eta(x) \d x  d s\,,
\eeqa
which shows that $u$ and $z$ fulfill \eqref{eq:Int} for a.\,e. $x\in D$. Since \eqref{eq:Int} is equivalent to \eqref{eq:weakmaterial1}, this finally proves \eqref{eq:weakmaterial1}.

{\it (iv) Convergence of the discrete image sequences pointwise everywhere in time.}
It remains to prove that for a subsequence of the discrete intensity functions $u^K$ (again denoted by $(u^K)_{K\in \N}$)  $u^K(t)\rightharpoonup u(t)$ weakly in $L^2(D)$ for all $t\in [0,1]$. To this end consider an arbitrary test function $\eta\in C^\infty_c(D)$, $t\in (0,1)$ and $\tau>0$ sufficiently small (in the what follows for $t=0$: $t-\tau$ is replaced by $t$ and for $t=1$: $t+\tau$ is replaced by $t$). Then, we obtain
\beqan
&&\int_D \left(\imagecurve^K(t,x)-\imagecurve(t,x)\right)\eta(x)\d x \nonumber\\
&&= \dashint_{t-\tau}^{t+\tau}\int_D\left(\imagecurve^K(t,x)-\imagecurve^K(s,x)\right)\eta(x)- \left(\imagecurve(t,x)-\imagecurve(s,x)\right)\eta(x)\d x\d s\notag\\
&&\quad +\; \dashint_{t-\tau}^{t+\tau}\int_D\left(\imagecurve^K(s,x)-\imagecurve(s,x)\right)\eta(x)\d x\d s\,. \label{eq:L2WeakConvergenceInTime}
\eeqan
Here, $\dashint_{t-\tau}^{t+\tau} f(s) \d s = \tfrac1{2\tau} \int_{t-\tau}^{t+\tau} f(s) \d s$ is the time-averaged integral of $f$ on $(t-\tau,t+\tau)$.
Due to the weak convergence of $u^K\rightharpoonup u$ in $L^2((0,1) \times D)$ the second integral on the right-hand side of \eqref{eq:L2WeakConvergenceInTime} 
vanishes as $K\rightarrow\infty$. Setting 
\[
\tilde\eta^K(t,y)=\eta(\psi^K(t,y))\det D\psi^K(t,y)\,, \quad  \tilde\eta(t,y)=\eta(\psi(t,y))\det D\psi(t,y)
\]
we can rewrite the first term in the first integral on the right-hand side of \eqref{eq:L2WeakConvergenceInTime} and get
\beqan
&&\dashint_{t-\tau}^{t+\tau}\int_D\imagecurve^K(t,\psi^K(t,y))\tilde\eta^K(t,y)-\imagecurve^K(s,\psi^K(s,x))\tilde\eta^K(s,y)\d y\d s \notag\\
&&=\dashint_{t-\tau}^{t+\tau}\int_D \left(\imagecurve^K(t,\psi^K(t,y))-\imagecurve^K(s,\psi^K(s,y))\right)\tilde\eta^K(t,y)\d y\d s \notag\\
&&\quad +\dashint_{t-\tau}^{t+\tau}\int_D\imagecurve^K(s,\psi^K(s,y))\left(\tilde\eta^K(t,y)- \tilde\eta^K(s,y)\right)\d y\d s\,.
\label{eq:eta}
\eeqan
The second integral on the right-hand side of \eqref{eq:eta} vanishes due to the smoothness of $\eta$ and $\psi^K$ as $\tau\rightarrow 0$.
Furthermore, using \eqref{eq:TimeBoundImages} the first integral can be estimated by   
\beqa
&&\left|\dashint_{t-\tau}^{t+\tau}\int_D \left(\imagecurve^K(t,\psi^K(t,y))-\imagecurve^K(s,\psi^K(s,y))\right)\tilde\eta^K(t,y)\d y\d s\right|\\
&& \leq \sup_{s\in[t-\tau, t+\tau]} 
\funcnorm{u^K(t, \psi^K(t,\cdot))-u^K(s, \psi^K(s,\cdot))}{L^2(D)}
\funcnorm{ \tilde\eta^K(t,\cdot)}{L^2(D)}\\
&& \leq C \tau^{\frac12} \funcnorm{ \tilde\eta^K(t,\cdot)}{L^2(D)}\,,
\eeqa
and thus also vanishes for $\tau \to 0$.
Analogous estimates apply to the remaining expression in \eqref{eq:L2WeakConvergenceInTime} replacing $\imagecurve^K$,
$\tilde\eta^K$, and $\psi^K$ by $\imagecurve$, $\tilde\eta$, and $\psi$, respectively. 
Altogether, this proves $u^K(t)\rightharpoonup u(t)$ weakly in $L^2(D)$ for all $t\in [0,1]$.\\[1ex]
{\it The limsup---estimate:}\medskip

{\it (i) Construction of the recovery sequence.}
Consider an image curve  $\imagecurve\in L^2((0,1)\times D)$.
Without any restriction we assume that the energy
\beq
\continuouspathenergy[u] = \int_0^1 \int_D L[v,v] + \frac1\delta |z|^2 \d x \d t
\eeq
is bounded, where $v \in L^2((0,1), W^{m,2}(D) \cap W^{1,2}_0(D))$ and $z \in L^2((0,1)\times D)$ are an optimal velocity field and a corresponding weak material derivative, respectively.
Now, we define an approximate, piecewise constant (in time) velocity field  
\begin{equation*} 
\left. v^K\right|_{[t_{k-1},t_k)} =v^K_k := K \int_{t_{k-1}}^{t_k} v \d t
\end{equation*}
for $k=1,\ldots,K$ and again denoting $t_k = \tfrac{k}{K}$. 

Obviously, $v^K$ converges to $v$ in $L^2((0,1),W^{m,2}(D))$. 
We denote by $\psi^K$ the associated flow of diffeomorphism generated by the flow equation $\dot \psi^K(t,x) = \tilde v^K(t,\psi^K(t,x))$
as in \eqref{eq:discreteODE} (for $\tilde v^K$ deduced from $\phi^K_k = \Id + K^{-1} v^K_k$)
with $\psi^K(0,x) = x$ and by $\psi^K_{t,s} = \psi^K(s,(\psi^K)^{-1}(t,\cdot))$ the induced relative deformation from time $t$ to time $s$. 
From this, we also obtain the underlying vector of consecutive deformations $\defvec^K = (\phi^K_1, \ldots, \phi^K_K)$ with $\phi^K_k = \psi^K_{t_{k-1},t_{k}}$.
Following \cite{DuGrMi98} we easily verify that the evolution equation for $\psi^K$, the uniform smoothness of $\psi^K$ and the bound on the energy $\continuouspathenergy[u]$ 
imply that $\psi^K$ is uniformly bounded in $C^{0,\frac12}([0,1],C^{1,\alpha}(\bar D))$ (cf. the proof of the $\liminf$-estimate above).

Next, the approximate discrete image path $\imagevec^K=(u^K_0,\ldots, u^K_K)$ is defined by a discrete counterpart of \eqref{eq:Int}, namely
\begin{equation} \label{eq:intK}
\imagecurve^K_k(x) = \imagecurve(0,\psi^K_{t_k,0}(x)) + \int_0^{t_k} z(s,\psi^K_{t_k,s}(x)) \d s
\end{equation}
for $k=0,\ldots, K$. 
Using \eqref{eq:ImagecurveRep} one obtains $u^K =  \interpolimage_{K}[\imagevec^K,\defvec^K]$ as the requested approximation of $u$ for given $K\in \N$.

{\it (ii) Proof of the the $\limsup$ inequality.}
At first, we verify that $\limsup_{K\to\infty} \continuouspathenergy_K[\imagecurve^K]\leq \continuouspathenergy[\imagecurve]$.
From the minimizing property of $\interpolimage_{K}[\imagevec^K,\defvec^K]$ we deduce
\begin{eqnarray*}
\continuouspathenergy_K[\imagecurve^K] = \pathenergy_{K}[\imagevec^K] \leq K \sum_{k=1}^K \int_D W(D\phi^K_k) + \gamma|D^m\phi^K_k|^2 + \frac1\delta |u^K_k\circ \phi^K_k - u_{k-1}^K|^2 \d x\,.
\end{eqnarray*}
 Using the Cauchy-Schwarz inequality we derive from \eqref{eq:intK}
\beqa
&&  \int_D  |u^K_k\circ \phi^K_k(x) - u^K_{k-1}(x)|^2 \d x 
=  \int_D \left| \int_{t_{k-1}}^{t_k} z(s,\psi^K_{t_{k-1},s}(x)) \d s \right|^2 \d x  \\
&&\leq  \frac1{K} \int_{t_{k-1}}^{t_k} \int_D |z(s,x)|^2 \det D (\psi^K_{t_{k-1},s})^{-1}(x) \d x \d s\\
&&\leq  \frac1{K}  \int_{t_{k-1}}^{t_k} \left(1+ CK^{-\frac{1}{2}}\right)\int_D |z(s,x)|^2 \d x \d s \,,
\eeqa
where we have taken into account the estimate $|1-\det D (\psi^K_{t_{k-1},s})^{-1}(x)| \leq C K^{-\frac12}$, which follows from the uniform bound for $\psi^K$ in 
$C^{0,\frac12}([0,1],C^{1,\alpha}(\bar D))$. 
Furthermore, we obtain via Taylor expansion and the consistency assumption (W3)
\beqa
&&  \int_D W(D\psi^K_{t_{k-1},t_k})+ \gamma |D^m\psi^K_{t_{k-1},t_k}|^2 \d x \\
&&\leq  \int_D  \frac{1}{2K^2}D^{2}W(\Id)(Dv^K_k,Dv^K_k)+ \frac{\gamma}{K^2}  |D^m v^K_k|^2\d x + C \int_D \frac{1}{K^{3}}|Dv^K_k|^3\d x \\
&&= \frac1{K^2}  \int_D L[v^K_k,v^K_k] \d x +\frac{C}{K^3}\int_D |Dv^K_k|^3\d x \,.
\eeqa
The definition of $v^K_k$ together with Jensen's inequality implies
$$\int_D L[v^K_k,v^K_k] \d x \leq K \int_D \int_{t_{k-1}}^{t_k} L[v,v] \d t \d x\,.$$
To estimate the remainder of the Taylor expansion we proceed as follows.
At first, we obtain
$$
\funcnorm{v_k^K}{C^1(\bar D)}^2 \leq C \sum_{l=1}^K \funcnorm{v_l^K}{W^{m,2}(D)}^2 \leq C K \int_0^1 \funcnorm{v(t,\cdot)}{W^{m,2}(D)}^2 \d t \leq C K
$$
using the Sobolev embedding theorem together with the Cauchy-Schwarz inequality and the boundedness of the energy $\continuouspathenergy[u]$.
Hence, $\max_{k=1,\ldots, K} \funcnorm{v^K_k}{C^1(\bar D)} \leq C K^{\frac12}$, which implies
\begin{eqnarray*}
\sum_{k=1}^K \int_D |Dv^K_k|^3\d x &\leq & \max_{k=1,\ldots, K} 
\funcnorm{v_k^K}{C^1(\bar D)} \sum_{k=1}^K \int_D \left( K \int_{t_{k-1}}^{t_k} D v(t,x) \d t \right)^2\d x \\
&\leq& C K^{\frac12} \frac{K^2}{K} \sum_{k=1}^K \int_D \int_{t_{k-1}}^{t_k} |D v(t,x)|^2 \d t \d x \leq C \, K^{\frac32} \,.
\end{eqnarray*} 
From these estimates we finally deduce
\beqa
\continuouspathenergy_K[\imagecurve^K] &\leq&
 \int_0^1 \int_D L[v,v] + \frac1\delta |z|^2 \d x \d t  + C K^{-\frac12}+ \frac{C}{\delta} K^{-\frac12} \,.
\eeqa
{\it (iii) Convergence of the discrete image sequences.}
We are still left to demonstrate that  $\imagecurve^K\rightarrow \imagecurve$ in $L^2((0,1)\times D)$.
To see this, we first observe that by the theorem of Arzel{\`a}-Ascoli  and after selection of  a subsequence 
 $\psi^K$ converges to $\psi$ in $C^{0,\beta}([0,1], C^{1,\alpha}(\bar D))$ with $\beta < \frac12$ and $\alpha < m-\tfrac{d}{2}-1$.
From this and the quantitative control of the inverse of the diffeomorphisms  (\cf \cite[Lemma 9]{TrYo05a}) we deduce that 
$\psi^K_{t,s}$, its inverse, and also $D \psi^K_{t,s}$ converge  uniformly in $x$, $t$, and $s$. Thus, we get that for every $t\in(0,1)$ 
\beqa
\funcnorm{z(\cdot, \psi^K_{t, \cdot}(\cdot)) - z(\cdot, \psi_{t,\cdot}(\cdot))}{L^2((0,1)\times D)} \to 0\,, \quad 
\funcnorm{u(0,\psi^K_{t,0}(\cdot))-u(0,\psi_{t,0}(\cdot))}{L^2(D)} \to 0
\eeqa
for $K\to\infty$. Indeed, in case of the first  claim we argue as follows. 
Due to the uniform bound on $z$ in $L^2((0,1)\times D)$ we only have to show that
$\int_0^1 \int_D  z(s, \psi^K_{t, s}(x))^q \eta(s,x) \d x \d s$ converges to 
$\int_0^1 \int_D  z(s, \psi_{t,s}(x))^q  \eta(s,x) \d x \d s$ for all $\eta \in C^\infty_c((0,1)\times D)$ and $q=1,2$.
This is easily seen via integral transform, \ie
\beqa
&&\int_0^1 \int_D  z(s, \psi^K_{t, s}(x))^q \eta(s,x)  - z(s, \psi_{t,s}(x))^q   \eta(s,x) \d x \d s \\
&& = \int_0^1 \int_D  z(s, y)^q  \Big( \eta(s, (\psi^K_{t,s})^{-1}(y))(\det D\psi^K_{t,s})^{-1} (\psi^K_{t,s})^{-1}(y)  \\ && \qquad \qquad \qquad \qquad -
 \eta(s, (\psi_{t,s})^{-1}(y))(\det D\psi_{t,s})^{-1} (\psi_{t,s})^{-1}(y) 
 \Big)  \d y \d s\,,
\eeqa 
where the right-hand side converges to $0$ for $K \to \infty$. The argument for $u(0,\cdot)$ is analogous.
Hence, we can pass to the limit on the right-hand side of \eqref{eq:intK} and achieve
in analogy to the corresponding argument in the proof of the $\liminf$-estimate 
\beqa
 &&  \left((t,x) \mapsto u(0,\psi^K_{t,0}(x)) + \int_0^t z(s, \psi^K_{t, s}(x) ) \d s \right)\\
&&\rightarrow \left((t,x)\mapsto u(0,\psi_{t,0}(x)) + \int_0^t z(s, \psi_{t, s}(x)) \d s \right) = u\,,
\eeqa
\vspace{-2ex}
where the convergence is in $L^2((0,1)\times D)$. From this the claim follows easily.
\end{proof}

\begin{theorem}[Convergence of discrete geodesic paths]\label{thm:convergence}
Let $u_A,\,u_B \in L^2(D)$ and suppose that (W1-3) holds.
Furthermore, for every $K\in \N$, let $\imagecurve^K$ be a minimizer of $\continuouspathenergy_K$ subject to $u^K(0) = u_A,\, u^K(1) = u_B$.
Then, a subsequence of $(\imagecurve^K)_{K\in \N}$ converges weakly in $L^2((0,1) \times D)$ to a minimizer of the continuous path energy $\continuouspathenergy$ and the associated sequence of discrete energies converges to the minimal continuous path energy.
\end{theorem}
\begin{proof}
The proof is standard 
in $\Gamma$-convergence theory (cf. \cite{Br02}). 
Choosing $u^K_k = \tfrac{k}{K} u_B+  (1-\tfrac{k}{K}) u_A$ and $\phi^K_k = \Id$ we obtain an a priori bound for the discrete energy 
$\continuouspathenergy_K$, which implies an a priori bound for $z^K$ in $L^2((0,1)\times D)$. Using \eqref{eq:discreteInt}, the strong convergence of $\psi^K$, and the Cauchy-Schwarz inequality 
we get that for the $u^K_k$, which are associated with the minimizer of $\continuouspathenergy_K$, the estimate
$\|u^K_k\|^2_{L^2(D)} \leq C ( \|u_A\|_{L^2(D)}^2 + \tfrac{k}{K} \|z^K\|_{L^2((0,1)\times D)}^2)$ holds.
From this we deduce that $\imagecurve^K$ is uniformly bounded in $L^\infty((0,1),L^2(D))$.
Hence, there exists a  subsequence, again denoted by $(\imagecurve^K)_{K\in \N}$, with $\imagecurve^K \rightharpoonup \imagecurve$ (weakly) in $L^2((0,1)\times D)$ to some $\imagecurve\in L^2((0,1)\times D)$.
Now, let us assume that there is an image path $\tilde \imagecurve$ with 
$\continuouspathenergy[\tilde \imagecurve] < \continuouspathenergy[u]$.
Then, by the $\limsup$-estimate of Theorem \ref{thm:gamma} there exists a  sequence
$(\tilde \imagecurve^K)_{K\in \N}$ with $\tilde \imagecurve^K \in L^2((0,1)\times D)$ such that $\limsup_{K\to \infty} \continuouspathenergy_K[\tilde \imagecurve^K] \leq 
\continuouspathenergy[\tilde \imagecurve]$ and together with the  $\liminf$-estimate we obtain
\beq
\continuouspathenergy[\imagecurve] \leq \liminf_{K\to \infty} \continuouspathenergy_K[ \imagecurve^K] \leq \limsup_{K\to \infty} \continuouspathenergy_K[\tilde\imagecurve^K] \leq 
\continuouspathenergy[\tilde \imagecurve]\,,
\eeq
which is a contradiction. Hence, $u$ minimizes the continuous path energy
over all admissible image paths. 
\end{proof}
\begin{remark}[Inherited smoothness]
Continuous solutions $\imagecurve$  of the metamorphosis model inherit for all $t\in (0,1)$ the regularity of the 
input images $u_A$ and $u_B$ (up to the H\"older regularity for the exponent $\alpha$). This can be seen as follows.
For a minimizer of the continuous path energy on $L^2((0,1) \times D)$ with $u(0)  \in \imagespace$ and $u(1) \in \imagespace$,
Trouv\'e and Younes give in \cite[Theorem 4]{TrYo05a} and \cite[Theorem 2]{TrYo05a} a direct representation of the intensity function, namely 
$$u(t,\cdot) = u(0, \psi(t)^{-1}(\cdot)) + \left(z_0 \int_0^t (\det D \psi(s))^{-1} \d s \right)  \circ \psi(t)^{-1}$$ 
for some $z_0 \in L^2(D)$ and $\psi(t) = \psi(t,\cdot)$ the underlying flow of diffeomorphisms.
Now, evaluating this equation for $t=1$ gives 
$$z_0 =(u(1,\psi(1,\cdot)) - u(0)) \left(\int_0^1 (\det D \psi)^{-1}(s) \d s \right)^{-1}\,.$$
Hence, $z_0$ is as regular as $u_A$ and $u_B$ (up to the H\"older regularity for the exponent $\alpha$) 
and the same holds true for $u(t,\cdot)$ for all $t\in [0,1]$.
For discrete solutions $\imagecurve^K$, the analog statement is already given in Remark \ref{remarklinear}.
\end{remark}

\section{Spatial discretization} 
We consider a regular quadrilateral grid on a two-dimen\-sional, rectangular  image domain $D$ consisting of cells $\mathcal{C}_m$ with $m\in I_C$, where $I_C$ is the index set of all cells. Based on this grid, we define the finite element space $\V_h$ of piecewise bilinear continuous functions (cf. \cite{Br07}) and denote by
$\set{\basisfct^i}_{i\in I_N}$ the set of basis functions, where $I_N$ is the index set of all grid nodes $x_i$.
Now, we investigate spatially discrete deformations $\Phi_k: D\to D$ with $\Phi_k\in \V_h^2$ ($k=1,\ldots, K$) and spatially discrete image maps $\Image_k:D \to \R$ ($k=0,\ldots, K$) with  $\Image_k\in \V_h$ and $\Image_0 = \Image_A = \Ih \image_A,\quad \Image_K = \Image_B = \Ih \image_B$. Here, $\Ih$ denotes the nodal interpolation operator. Given a finite element function $W \in \V_h$,
we denote by $\bar W = (W(x_i))_{i\in I_N}$ the corresponding vector of nodal values.
Furthermore, we define a fully discrete counterpart $\pathenergy_{K,h}$ of the so far solely time discrete path energy $\pathenergy_{K}$
as follows
\begin{align*}
& \pathenergy_{K,h}[(\Image_0,\ldots, \Image_K)]  :=  \min_{\substack{\Phi_k \in \V_h^2, \, \Phi_k|_{\partial D}= \Id,\\ k=1,\ldots, K}}
\ESPhi{K,h}[(\Image_0,\ldots, \Image_K),(\Phi_1,\ldots, \Phi_K)]\,.
\end{align*}
Here, $\ESPhi{K,h}[(\Image_0,\ldots, \Image_K),(\Phi_1,\ldots, \Phi_K)]$ is the discrete counterpart of $\ESPhi{K}$ and obtained by approximating the integrals of $\ESPhi{K}$ on each cell with the Simpson quadrature rule. Here, the standard 3-point Simpson quadrature rule in 1D is extended to 2D with 9 points using the tensor product. In our numerical experiments, this $9$--point quadrature rule performed well. 
In particular, compared to lower order quadrature rules, it avoids blurring effects in the vicinity of image edges. Let us remark that due to the concatenation with the deformation 
an exact integration with standard quadrature rules is not possible.

Next, we study the numerical minimization of the fully discrete energy $\ESPhi{K,h}$ for fixed $(\Phi_1,\ldots, \Phi_K)$.
For $m\in I_C$  the Simpson quadrature takes into account nine quadrature points. Let $x_q^m$ denote the $q$-th quadrature point in $\mathcal{C}_m$ and $w_q^m$ the corresponding quadrature weight for $q\in \set{0,\ldots, 8}$.
Then, the entries of the weighted mass matrix $\mass_h[\Phi,\Psi]= \left(\mass_h[\Phi,\Psi]_{i,j}\right)_{i,j\in I_N}$
with basis functions being transformed via deformations $\Phi, \Psi$ and evaluated via quadrature are given by
\begin{align*}
\mass_h[\Phi,\Psi]_{i,j} :={}& \sum_{l\in I_C}\sum_{q=0}^8 w_q^l(\basisfct^i\circ \Phi)(x_q^l)\,(\basisfct^j\circ \Psi)(x_q^l)\,.
\end{align*}
To evaluate the entries of this matrix numerically, we use cell-wise assembly. 
For $m\in I_C$, let $\basisfct_\alpha^m$ denote the basis function in the cell $\mathcal{C}_m$ with local index $\alpha\in \set{0,1,2,3}$ and $I(m,\alpha)$ the global
index corresponding to the local index $\alpha$ in the cell $\mathcal{C}_m$, \ie $\basisfct_{I(m,\alpha)} = \basisfct_\alpha^m$ on $\mathcal{C}_m$.
The cell-wise assemble procedure works as follows. First, $\mass_h[\Phi,\Psi]$ is initialized as the zero matrix. Then, for the every $l\in I_C$ and every $q\in \set{0,\ldots, 8}$ one
identifies the cells $\mathcal{C}_m$, $\mathcal{C}_{m'}$ with
$\Phi(x^l_q) \in \mathcal{C}_m$ and $\Psi(x^l_q) \in \mathcal{C}_{m'}$, respectively.
Finally, for all pairs of local indices $(\beta,\beta')$ with $\beta, \beta' \in \set{0,1,2,3}$ one adds $w_q^l \basisfct^m_\beta(\Phi(x^l_q)) \basisfct^{m'}_{\beta'}(\Psi(x^l_q))$
to $\mass_h[\Phi,\Psi]_{I(m,\beta),I(m',\beta')}$. 

Now, we are in the position to derive a linear system of equations for the vector $\bar \U = (\bar U_1,\ldots, \bar U_{K-1})$ of images
that describes a minimizer of $\ESPhi{K,h}$ for a fixed vector of spatially discrete deformations $\bar{\mathbf{\Phi}} =  (\bar \Phi_1,\ldots, \bar \Phi_{K})$.
Indeed, we can rewrite the last term in the energy $\ESPhi{K,h}$ as follows
\begin{align*}
&\sum_{k=1}^K\sum_{l\in I_C}\sum_{q=0}^8 w_q^l \left(|\Image_{k} \circ \Phi_k-\Image_{k-1}|^2\right)(x_q^l)\\
&=\sum_{k=1}^K \left(\mass_h[\Phi_k,\Phi_k] \bar \Image_k\cdot \bar \Image_k - 2\mass_h[\Phi_k,\Id] \bar \Image_k\cdot \bar \Image_{k-1}
+ \mass_h[\Id,\Id] \bar \Image_{k-1}\cdot \bar \Image_{k-1} \right).
\end{align*}
From this, we obtain for the variation of the energy  $ \ESPhi{K,h}$ with respect to the $k$-th image map
$$
\partial_{\bar \Image_k}  \ESPhi{K,h} = 2 \left( \mass_h[\Phi_k,\Phi_k] + \mass_h[\Id,\Id] \right) \bar \Image_k - 2 \mass_h[\Phi_k,\Id]^T \bar \Image_{k-1}
- 2 \mass_h[\Phi_{k+1},\Id] \bar \Image_{k+1}
$$
for $k=1,\ldots, K-1$. In the semi-Lagrangian approach for the flow of diffeomorphisms model, a similar computation appears in the context of the single matching penalty with respect to the given end image (cf. \cite{BeMiTr05}). For a fixed set of deformations a necessary condition for $\bar \U$ to be a minimizer of $\ESPhi{K,h}$ is that $\bar \U$ solves 
the block tridiagonal system of linear equations $\MatrixBF[\mathbf{\Phi}] \bar  \U = \Right[\mathbf{\Phi}]$, where
$\MatrixBF[\mathbf{\Phi}]$ is formed by $(K-1)\times(K-1)$ matrix blocks $\MatrixBF_{k,k'} \in  \R^{I_N\times I_N}$ and $\Right[\mathbf{\Phi}]$ consists of $K-1$ vector blocks $\Right_k \in \R^{I_N}$
with
\begin{align*}
& \MatrixBF_{k,k-1} =  -\mass_h[\Phi_{k},\Id]^T\,,\;
\MatrixBF_{k,k} =  \mass_h[\Phi_{k},\Phi_{k}]+\mass_h[\Id,\Id]\,,\;
\MatrixBF_{k,k+1} = -\mass_h[\Phi_{k+1},\Id] \,,\\
& \Right_1 =\mass_h[\Phi_{1},\Id]^T\bar \Image_A \,,\;\Right_2 = \Right_3 = \ldots = \Right_{K-2} = 0\,,\;
\Right_{K-1} =  \mass_h[\Phi_{K},\Id]\bar \Image_B \,.
\end{align*}
The energy $\sum_{l\in I_C}\sum_{q=0}^8 w_q^l \left(|\Image_{k} \circ \Phi_k-\Image_{k-1}|^2\right)(x_q^l)$
is convex in $\Image_{k}$ (as a quadratic function of convex combinations of components of $\Image_{k}$) and strictly convex in $\Image_{k-1}$. 
Here, we use that the quadrature rule integrates affine functions exactly.
Hence, $\ESPhi{K,h}$ is strictly convex in $\U$ and
there is a unique minimizer $\U = \U[\mathbf{\Phi}]$ for fixed $\mathbf{\Phi}$. This implies that $\MatrixBF$ is invertible and by solving the
linear system $\MatrixBF[\mathbf{\Phi}] \bar  \U = \Right[\mathbf{\Phi}]$ one computes this unique minimizer.
Numerically, the corresponding system of linear equations (cf. line \ref{algo:LinearSystem} of Algorithm~\ref{alternatingAlgorithm}) is solved with a conjugate gradient method with diagonal preconditioning.

For fixed $\U$, the deformations $\Phi_1,\ldots,\Phi_K$ are independent of each other and thus can be updated separately.
In the case of bilinear finite elements, we consider only even $m$ and replace the integrand $|D^m v|^2$ by $|\Delta^{\frac{m}{2}} v|^2$ in the quadratic form \eqref{definitionEllipticOperator} and correspondingly $|D^m \phi|^2$ by  
$|\Delta^{\frac{m}{2}} \phi|^2$ in the energy \eqref{eq:WEnerDefinition}. 
By elliptic regularity theory, all of the results above directly transfer to this modified functional.
Furthermore, we use the same quadrature rule as before for the elastic energy and obtain the fully discrete energy 
\begin{eqnarray*}
\ESPhi{K,h}[(\Image_0,\ldots, \Image_K),(\Phi_1,\ldots, \Phi_K)]  &=&
 \sum_{k=1}^K  \Big( \sum_{l\in I_C}\sum_{q=0}^8 w_q^l W(D\Phi_k(x^l_q))  \\
&& \qquad \; + \gamma \, \sum_{n=1,2} \mass_h (\mass_h^{-1} \stiff_h)^{\frac{m}{2}} \bar \Phi^n_k \cdot (\mass_h^{-1} \stiff_h)^{\frac{m}{2}} \bar \Phi^n_k  \\
&& \qquad \;  + \frac1\delta \sum_{l\in I_C}\sum_{q=0}^8 w_q^l \left(|\Image_{k} \circ \Phi_k-\Image_{k-1}|^2\right)(x_q^l) \Big)\,,
\end{eqnarray*}
where
$\stiff_h[\Phi,\Psi]_{i,j} := \sum_{l\in I_C} \sum_{q=0}^8 w_q^l \nabla\Theta^i(x^l_q)\cdot\nabla\Theta^j(x^l_q)$  is the stiffness matrix and $\Phi^n_k$ the $n$-th component of $\Phi_k$.
The actual minimization of $\ESPhi{K,h}$ with respect to $\Phi_k$ (the numerical solution of a simple registration problem) is implemented based on a step size controlled 
Fletcher-Reeves nonlinear conjugate gradient descent scheme with respect to a regularized $H^1$-metric on the space of deformations \cite{SuYeMe07}. 
Thereby, the gradient of the energy $\ESPhi{K,h}$ with respect to the deformation $\Phi_k$ in a direction $\Theta$ is given by
\begin{eqnarray*}
&&<\partial_{\Phi_k}\ESPhi{K,h}[(\Image_0,\ldots, \Image_K),(\Phi_1,\ldots, \Phi_K)],\Theta>  = \\
&&   \sum_{k=1}^K  \Big( \!\sum_{l\in I_C}\sum_{q=0}^8 w_q^l W_{,A}(D\Phi_k(x^l_q))(D \Theta(x^l_q))
 + 2 \gamma \!  \! \sum_{n=1,2}  \! \mass_h (\mass_h^{-1} \stiff_h)^{\frac{m}{2}} \bar \Phi^n_k \cdot (\mass_h^{-1} \stiff_h)^{\frac{m}{2}} \bar \Theta^n  \\
&&  \qquad \; \; + \frac2\delta  \sum_{l\in I_C}\sum_{q=0}^8 w_q^l \left(\Image_{k} \circ \Phi_k-\Image_{k-1}\right)(x_q^l) \left((\nabla \Image_{k}\circ \Phi_k) \cdot \Theta\right)(x_q^l)\Big) \,.
\end{eqnarray*}
Furthermore, we take into account a cascadic approach starting with a coarse time discretization and then successively refine the time discretization. In each step of this approach, we
minimize the discrete path energy and perform a prolongation to the next finer level of the time discretization. The prolongation is based on the insertion of new midpoint images between every pair of consecutive images. To this end, we compute an optimal deformation between a pair of images and insert the middle image of the resulting warp.
To improve the robustness of the algorithm, we additionally use a Gaussian filter with variance $\sigma^2 = \frac54 h$ (color images) or $\sigma^2 = \frac58 h$ (black-and-white images) to pre-filter the input images and damp noise, where $h$ is the mesh size.
The resulting alternating minimization algorithm is summarized in Algorithm \ref{alternatingAlgorithm}.
\begin{algorithm}[t]
\small 
\KwData{input images $\Image_A$ and $\Image_B$}
\KwResult{approximate minimizer $(\Image_A=\Image_0^J,\Image_1^J,\ldots,\Image_K^J=\Image_B)$ of $\pathenergy_{K}$}
smooth $\Image_0^0=\Image_A$ and $\Image_1^0=\Image_B$ with the Gaussian filter with variance $\sigma^2$\;
\For{$j=1$ \KwTo $J$}{
$K=2^{j}$\;
$\Image_{2k}^{j}=\Image_{k}^{j-1}$ for $k=0,1,\ldots,K/2$\;
\For{$k=0$ \KwTo $K/2-1$}{
calculate  $\Phi\in\argmin_{\tilde\Phi\in\V_h^2}\ESPhi{K}[(\Image_{2k}^{j},\Image_{2k+2}^{j}),\tilde\Phi]$\;
$\Image_{2k+1}^{j}=\Image_{2k+2}^{j}\circ(\Id+0.5(\Phi-\Id))$\;
}
\Repeat{$\funcnorm{\bar\U^{j,old}-\bar\U^{j}}{2}\leq {threshold}$}{
$\bar\U^{j,old}=(\Image_1^{j},\ldots, \Image_{K-1}^{j})$\;
compute $\mathbf{\Phi}^{j}=(\Phi_1^{j},\ldots, \Phi_{K}^{j})\in\argmin_{\Phi\in (\V_h^2)^{K}}
\ESPhi{K}[(\Image_A,\bar\U^j,\Image_B),\Phi]$\;
calculate $\bar\U^{j}=(\Image_1^{j},\ldots, \Image_{K-1}^{j})$ via
$\bar  \U^{j} = \Matrix[\mathbf{\Phi}^{j}]^{-1} \Right[\mathbf{\Phi}^{j}]$ \;
\label{algo:LinearSystem}
}
}
\vspace{1ex}

\caption{The alternating gradient descent scheme to compute the geodesic path.}
\label{alternatingAlgorithm}
\end{algorithm}
In the applications, it is frequently appropriate to ensure that deformations are not restricted too much by the
Dirichlet boundary condition $\Phi=\Id$ on $\partial D$. This can practically be obtained by enlarging the computational domain and considering an extension of the image intensities with a constant gray or color value or by taking into account
natural boundary conditions for the deformations. This can theoretically be justified by adding constraints on the
mean deformation and the angular momentum. In our computations, such constraints are usually not required to avoid an unbounded rigid body motion component of the numerical solution.

\section{Numerical Results}
\setlength{\unitlength}{0.05\textwidth}
\begin{figure}[t]
\resizebox{\linewidth}{!}{
\begin{picture}(21,18)
\put(0,18){\footnotesize{$K=4$,}}
\put(0,17.4){\footnotesize{using \eqref{eq:WEnerDefinition}, \eqref{eq:energyDensityExOne}}}
\put(0,16.8){\footnotesize{with $\delta=10^{-2}$,}}
\put(0,16.2){\footnotesize{$\lambda=1$, $\mu=\frac12$,}}
\put(0,15.6){\footnotesize{ $q=r=\frac32$,}}
\put(0,15){\footnotesize{ $s=\frac12$, $\gamma=10^{-5}$}}
\put(3.5,15){\includegraphics[width=0.165\textwidth]{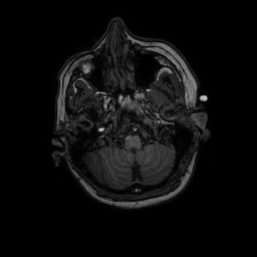}}
\put(7,15){\includegraphics[width=0.165\textwidth]{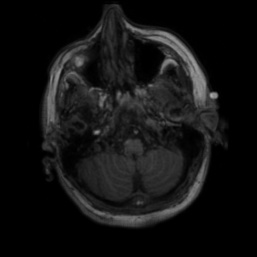}}
\put(10.5,15){\includegraphics[width=0.165\textwidth]{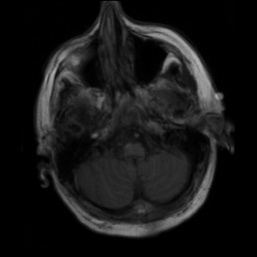}}
\put(14,15){\includegraphics[width=0.165\textwidth]{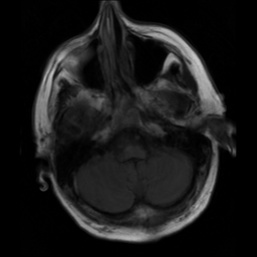}}
\put(17.5,15){{\includegraphics[width=0.165\textwidth]{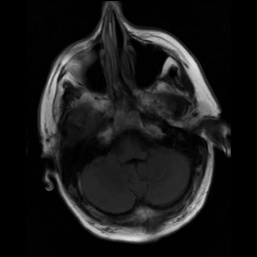}}}
\put(0,14.65){\line(1,0){21}}
\put(0,13.6){\footnotesize{$K=4$,}}
\put(0,13){\footnotesize{using \eqref{eq:WEnerDefinitionsimple}} with}
\put(0,12.4){\footnotesize{$\gamma=10^{-3}$,}}
\put(0,11.8){\footnotesize{$\delta=10^{-1}$}}
\put(3.5,11){\includegraphics[width=0.165\textwidth]{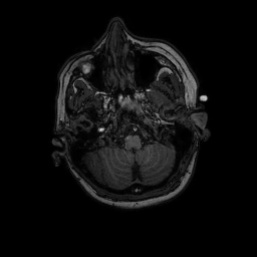}}
\put(7,11){\includegraphics[width=0.165\textwidth]{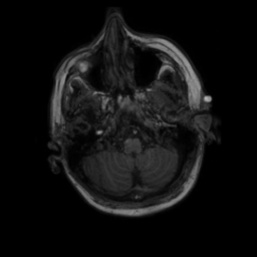}}
\put(10.5,11){\includegraphics[width=0.165\textwidth]{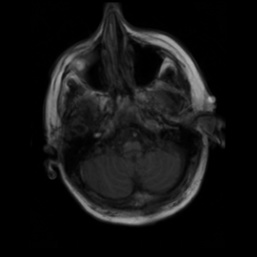}}
\put(14,11){\includegraphics[width=0.165\textwidth]{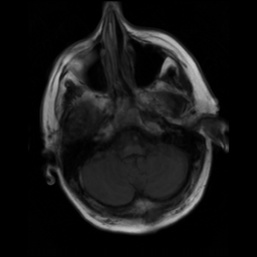}}
\put(17.5,11){{\includegraphics[width=0.165\textwidth]{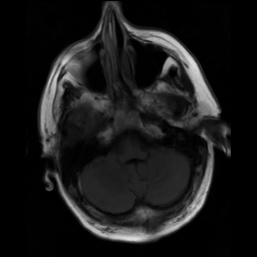}}}
\put(0,10.65){\line(1,0){21}}
\put(0,9.6){\footnotesize{$K=16$,}}
\put(0,9){\footnotesize{using \eqref{eq:WEnerDefinitionsimple}} with}
\put(0,8.4){\footnotesize{$\gamma=10^{-3}$,}}
\put(0,7.8){\footnotesize{$\delta=10^{-1}$}}
\put(3.5,7){\includegraphics[width=0.165\textwidth]{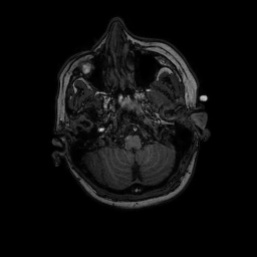}}
\put(7,7){\includegraphics[width=0.165\textwidth]{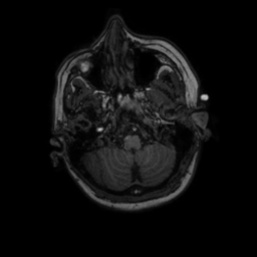}}
\put(10.5,7){\includegraphics[width=0.165\textwidth]{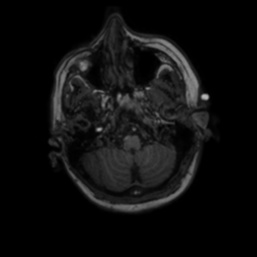}}
\put(14,7){\includegraphics[width=0.165\textwidth]{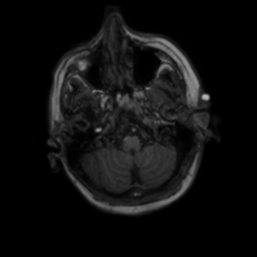}}
\put(17.5,7){\includegraphics[width=0.165\textwidth]{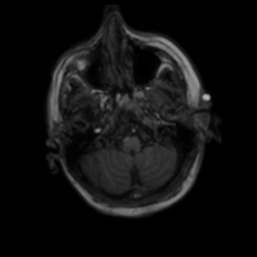}}
\put(0,3.5){\includegraphics[width=0.165\textwidth]{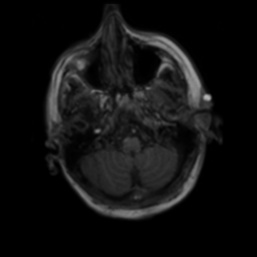}}
\put(3.5,3.5){\includegraphics[width=0.165\textwidth]{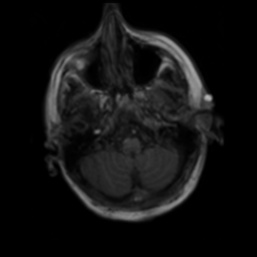}}
\put(7,3.5){\includegraphics[width=0.165\textwidth]{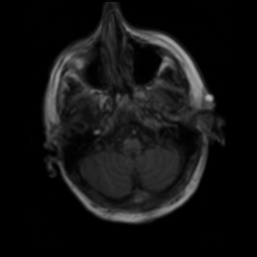}}
\put(10.5,3.5){\includegraphics[width=0.165\textwidth]{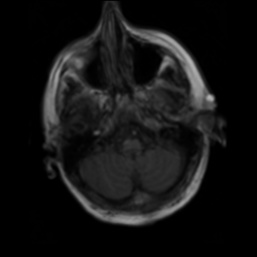}}
\put(14,3.5){\includegraphics[width=0.165\textwidth]{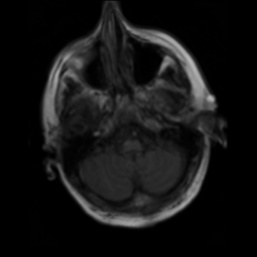}}
\put(17.5,3.5){\includegraphics[width=0.165\textwidth]{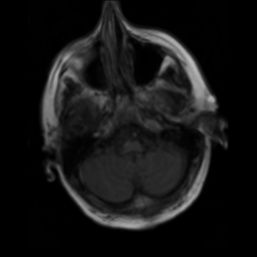}}
\put(0,0){\includegraphics[width=0.165\textwidth]{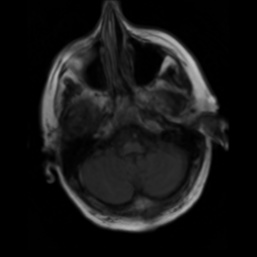}}
\put(3.5,0){\includegraphics[width=0.165\textwidth]{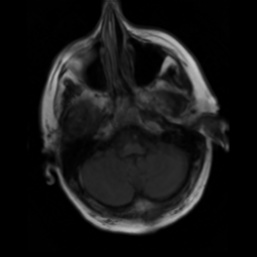}}
\put(7,0){\includegraphics[width=0.165\textwidth]{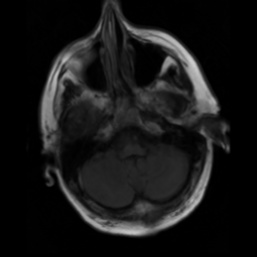}}
\put(10.5,0){\includegraphics[width=0.165\textwidth]{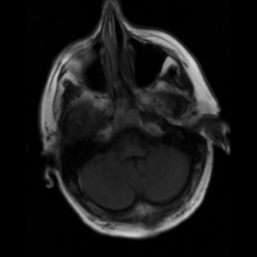}}
\put(14,0){\includegraphics[width=0.165\textwidth]{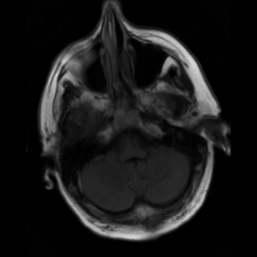}}
\put(17.5,0){\includegraphics[width=0.165\textwidth]{images/MRT-1/U_04_015}}
\end{picture}
}
\caption{Metamorphosis for two slices of a MRT data set  of a human brain (data courtesy of H. Urbach, Neuroradiology, University Hospital Bonn). We compare the original model (first row) with the simplified model and $K=4$ (second row), 
$K=16$ (third to fifth row). }
\label{MRT-1}
\end{figure}

In this section, we discuss numerical results for the metamorphosis model, which are obtained with Algorithm \ref{alternatingAlgorithm} proposed in the preceding section. Besides the original model \eqref{eq:WEnerDefinition}, 
we consider a simplified model, which gives result of comparable visual quality with less computational effort.
In the original model, we use as discussed above for $m=4$ instead of $|D^4 v|^2$ the term $|\Delta^2 v|^2$ in the quadratic form \eqref{definitionEllipticOperator}.
The simplified model is associated with the quadratic form 
\beqn
L[v(t),v(t)] :=   D v : Dv + \gamma  \Delta v \cdot \Delta v\,,
\label{definitionEllipticOperatorsimple}
\eeqn
where $\gamma >0$.
A choice for the discrete energy, which is consistent with this quadratic form, is given by 
\beqn
\WEner[\image,\tilde \image]= \min_{\phi} \int_{D} D \phi: D \phi +  \gamma \Delta \phi \cdot \Delta \phi
+\frac{1}{\delta} |\tilde \image \circ \phi-\image|^2 \d x\,.\label{eq:WEnerDefinitionsimple}
\eeqn 
In fact, this retrieves a very basic model for the registration of the two images
$\image$ and $\tilde \image$ consisting of a simple thin plate spline regularization and the most basic fidelity term (cf. \cite{MoFi03}).

Let us emphasize that in the spatially continuous setting both the existence theory and the $\Gamma$-convergence result require the full set of assumptions. In particular, the definitions \eqref{definitionEllipticOperatorsimple} for $L[\cdot,\cdot]$
and \eqref{eq:WEnerDefinitionsimple} for $\WEner$ (contrary to the full model with $W$ proposed in \eqref{eq:energyDensityExOne}) do not comply with (W2) and (W3). 
However, in case of the definition \eqref{eq:WEnerDefinitionsimple}, the regularization term of the deformation energy 
$\WEner$ is quadratic and enables a significant speedup of the algorithm compared to the theoretically justified fully nonlinear model. We compare both models in our first example and use the simplified model in all other applications.
The parameter $threshold$ is set to $10^{-6}$ in the algorithm.

Figure \ref{MRT-1} depicts a discrete geodesic path obtained with the full model (with parameters $K=4$, $\delta=10^{-2}$, $\lambda=1$, $\mu=\frac12$, $q=r=\frac32$, $s=\frac12$ and $\gamma=10^{-5}$) 
and with the simplified model (with parameters $K\in\{4,16\}$, $\gamma=10^{-3}$, $\delta=10^{-1}$), where $\image_A$ and $\image_B$ are different slices of a 3D magnetic resonance tomography of a human brain. 

Figure \ref{figure:womenExample} shows a geodesic path between two faces from female portrait paintings\footnote{first painting by A. Kauffmann (public domain, see \url{http://commons.wikimedia.org/wiki/File:Angelika_Kauffmann_-_Self_Portrait_-_1784.jpg}), second painting by R. Peale (GFDL, see \url{http://en.wikipedia.org/wiki/File:Mary_Denison.jpg})}
computed with the simplified model with parameters $\gamma=10^{-3}$ and $\delta=10^{-2}$. 
 The local contributions $\ESPhi{K}[(\Image_{k-1},\Image_{k}),\Phi_k]$  for $k=1,\ldots, K$ 
 of the total energy and its components are shown in Figure \ref{figure:womenExamplePlot}. 
 Note that the method seems to prefer an approximate equidistribution of the total path energy in time.
 
\setlength{\unitlength}{0.05\textwidth}
\begin{figure}[t]
\resizebox{\linewidth}{!}{
\begin{picture}(21,14)
\put(0,13.6){\footnotesize{$K=4$,}}
\put(0,13){\footnotesize{using \eqref{eq:WEnerDefinitionsimple}} with}
\put(0,12.4){\footnotesize{$\gamma=10^{-3}$,}}
\put(0,11.8){\footnotesize{$\delta=10^{-2}$}}
\put(3.5,11){\includegraphics[width=0.165\textwidth]{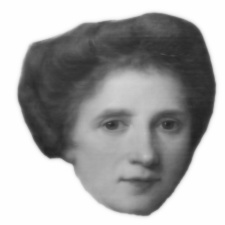}}
\put(7,11){\includegraphics[width=0.165\textwidth]{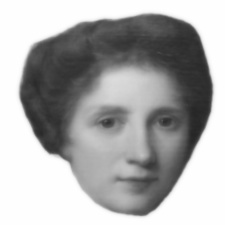}}
\put(10.5,11){\includegraphics[width=0.165\textwidth]{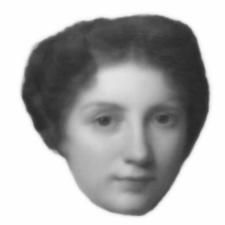}}
\put(14,11){\includegraphics[width=0.165\textwidth]{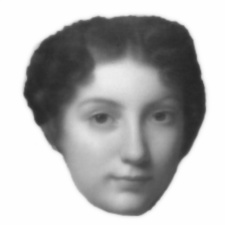}}
\put(17.5,11){{\includegraphics[width=0.165\textwidth]{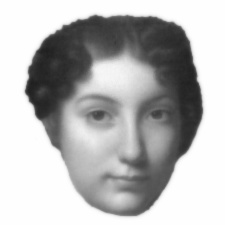}}}
\put(0,10.65){\line(1,0){21}}
\put(0,9.6){\footnotesize{$K=16$,}}
\put(0,9){\footnotesize{using \eqref{eq:WEnerDefinitionsimple}} with}
\put(0,8.4){\footnotesize{$\gamma=10^{-3}$,}}
\put(0,7.8){\footnotesize{$\delta=10^{-2}$}}
\put(3.5,7){\includegraphics[width=0.165\textwidth]{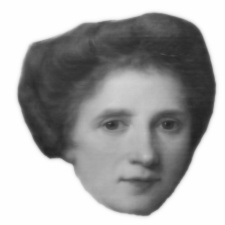}}
\put(7,7){\includegraphics[width=0.165\textwidth]{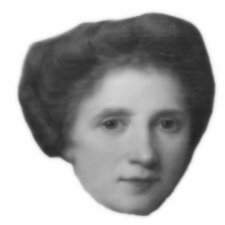}}
\put(10.5,7){\includegraphics[width=0.165\textwidth]{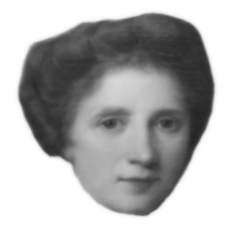}}
\put(14,7){\includegraphics[width=0.165\textwidth]{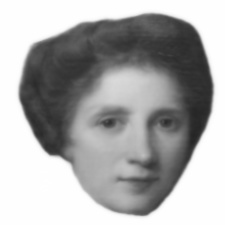}}
\put(17.5,7){\includegraphics[width=0.165\textwidth]{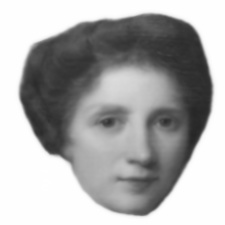}}
\put(0,3.5){\includegraphics[width=0.165\textwidth]{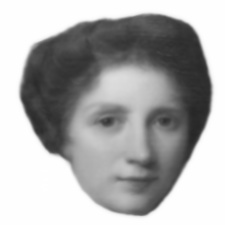}}
\put(3.5,3.5){\includegraphics[width=0.165\textwidth]{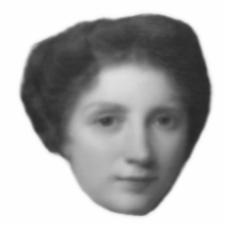}}
\put(7,3.5){\includegraphics[width=0.165\textwidth]{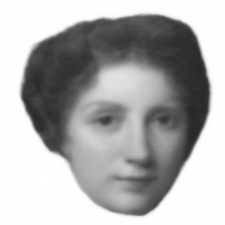}}
\put(10.5,3.5){\includegraphics[width=0.165\textwidth]{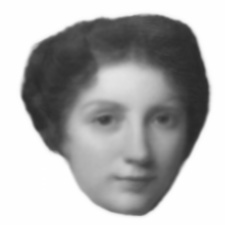}}
\put(14,3.5){\includegraphics[width=0.165\textwidth]{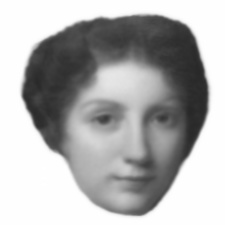}}
\put(17.5,3.5){\includegraphics[width=0.165\textwidth]{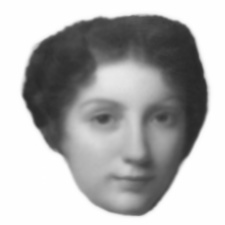}}
\put(0,0){\includegraphics[width=0.165\textwidth]{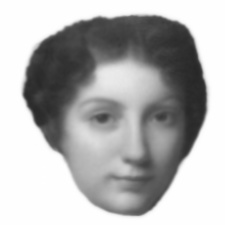}}
\put(3.5,0){\includegraphics[width=0.165\textwidth]{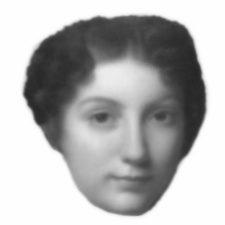}}
\put(7,0){\includegraphics[width=0.165\textwidth]{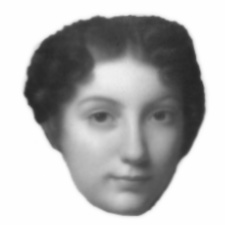}}
\put(10.5,0){\includegraphics[width=0.165\textwidth]{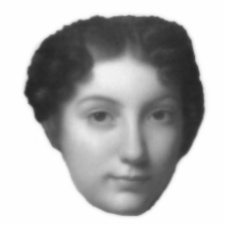}}
\put(14,0){\includegraphics[width=0.165\textwidth]{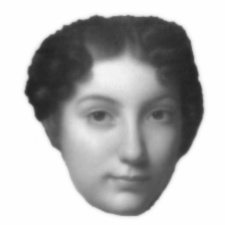}}
\put(17.5,0){\includegraphics[width=0.165\textwidth]{images/K/U_04_015_cropped}}
\end{picture}
}
\caption{Metamorphosis between two faces from female portrait paintings.}
\label{figure:womenExample}
\end{figure}

\begin{figure}[tb]
\begin{center}
\subfigure{\includegraphics[width=0.35\textwidth]{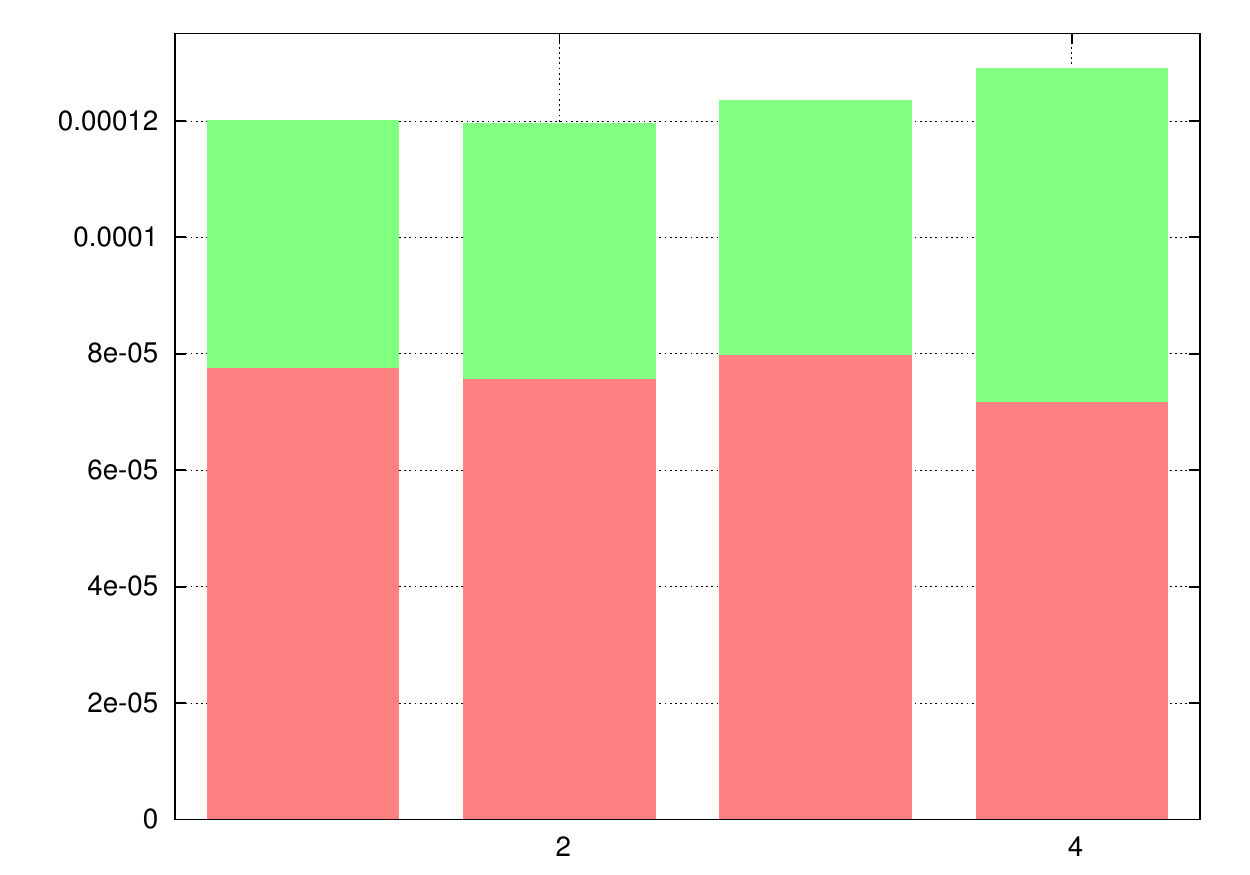}}
\subfigure{\includegraphics[width=0.35\textwidth]{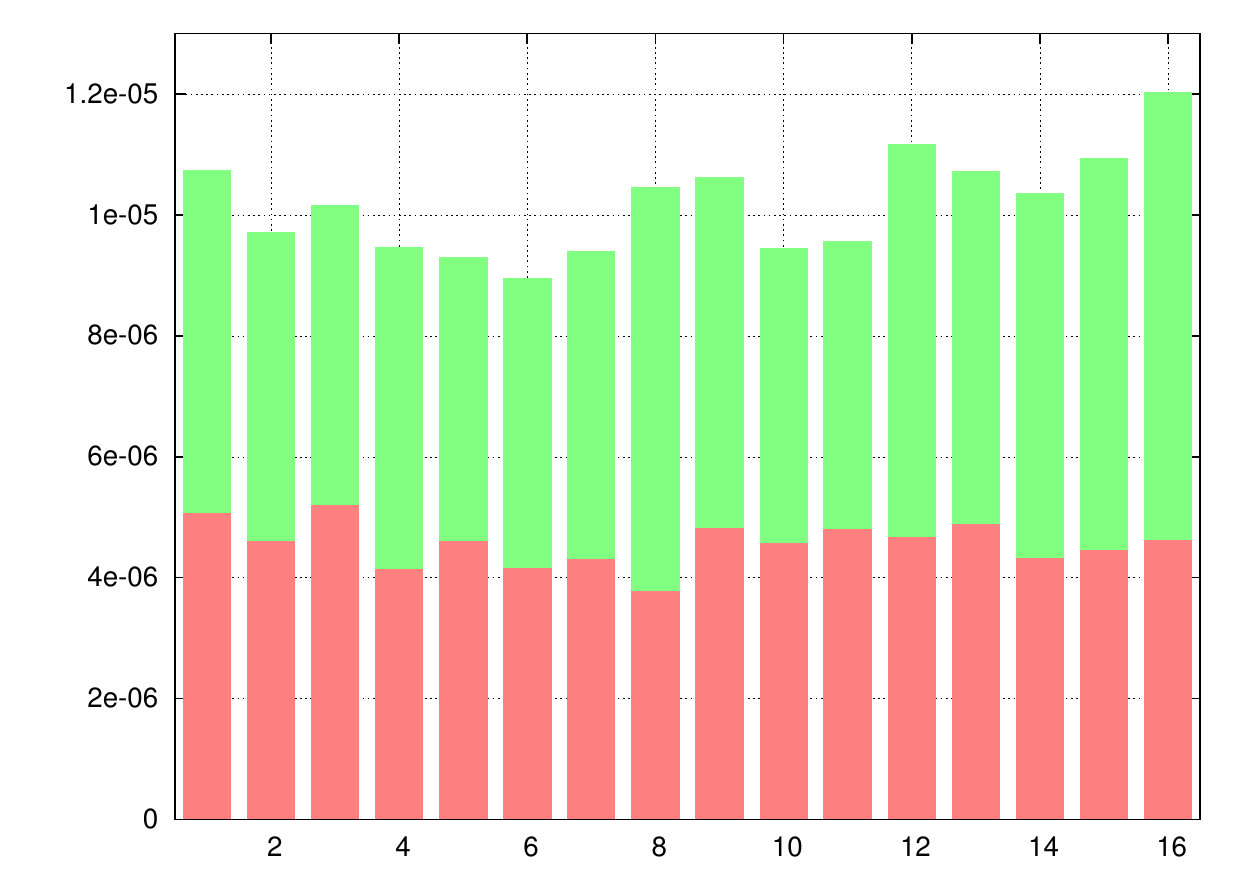}}\\[-2ex]
\caption{Energy contributions of the regularization functional 
$\int_{D} D \Phi_k: D \Phi_k +  \gamma \Delta \Phi_k \cdot \Delta \Phi_k \d x$ (red) and 
the matching functional $\frac{1}{\delta} \int_{D} |\tilde U_k \circ \Phi_k-U_{k-1}|^2 \d x$ (green) for the discrete geodesic path in 
 Figure \ref{figure:womenExample} with $K=4$ (left) and $K=16$ (right).}
\label{figure:womenExamplePlot}
\end{center}
\end{figure}

Finally, we consider time discrete geodesic paths in the space of color images. To this end, we take into account a straightforward generalization of the model for scalar (gray) valued image maps to vector-valued image maps. One can even enhance the model with further channels. Such additional channels  can represent segmented regions of the images, which one would like to ensure to be properly matched by transport and not by blending of intensities.
The only required modification of the method is that $|\image_{k+1} \circ \phi_{k+1}-\image_k|$ is now the Euclidean norm of the (extended) color vector. As an application, we considered the metamorphosis between two self-portraits by van Gogh (see Figure \ref{figure:vanGogh}) 
\footnote{both paintings by V. van Gogh (public domain, \url{http://en.wikipedia.org/wiki/File:SelbstPortrait_VG2.jpg},
\url{http://upload.wikimedia.org/wikipedia/commons/7/71/Vincent_Willem_van_Gogh_102.jpg})}.
Since the background colors of both self-portraits differ considerably in the RGB color space, we adjusted the background color of one of the images (\ie replacing $\tilde\image_B$ by $\image_B^{RGB}$ in Figure \ref{figure:vanGogh}). 
In this application, a fourth (segmentation) channel 
is used to ensure the proper 
\begin{wrapfigure}{r}{0.3\textwidth}
\vspace{-4ex}
\begin{center}
\includegraphics[width=0.8\linewidth]{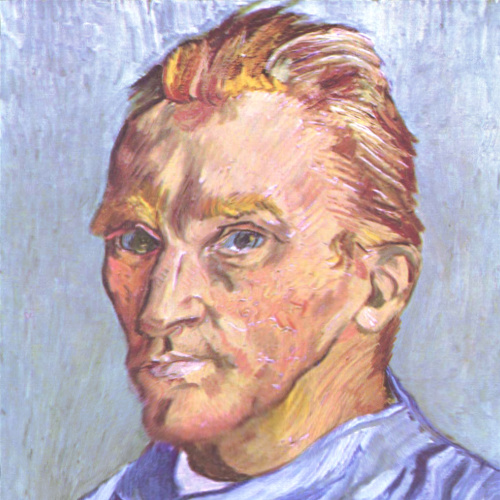}
\end{center}
\vspace{-4.5ex}
\caption{Pullback $\image_B^{RGB}\circ \Phi$ of image map $\image_B^{RGB}$ along the path.}
\vspace{-2ex}
\label{figure:vanGoghPullback}
\end{wrapfigure}
matching of the ears and the clothing.
The time-discrete geodesic path for the van Gogh self-portraits is shown 
in Figure \ref{figure:vanGoghGeodesicPath}
for $K=8$ along with the temporal change of the fourth channel. Again, we used the simplified model with parameters $\gamma=10^{-3}$ and $\delta=10^{-2}$.
Figure \ref{figure:vanGoghPullback} depicts the pullback $\image_B^{RGB}\circ\Phi$ along the flow induced deformation  
$\Phi=\Phi_K\circ\Phi_{K-1}\circ\ldots\circ\Phi_1$ corresponding
 to the geodesics in Figure \ref{figure:vanGoghGeodesicPath}.
Finally, Figure \ref{vzComponentVanGogh} visualizes the deformations and the corresponding accumulated weak material derivative along the discrete geodesic path. The color wheel on the lower left in the first row indicates both the direction and the magnitude of the discrete velocities $K (\Phi_{k}-\Id)$. Obviously, the motion field is not constant in time. Furthermore, to visualize the change of the image intensity along motion paths, the accumulated weak material derivative $Z_l$ ($l=1,\ldots,8$) with
$Z_l=K\sum_{k=1}^l (U_{k}\circ\Phi_{k}-U_{k-1})\circ X_{k-1}$ using the notation \eqref{eq:Xk}
is plotted using an equal rescaling for all $l$.
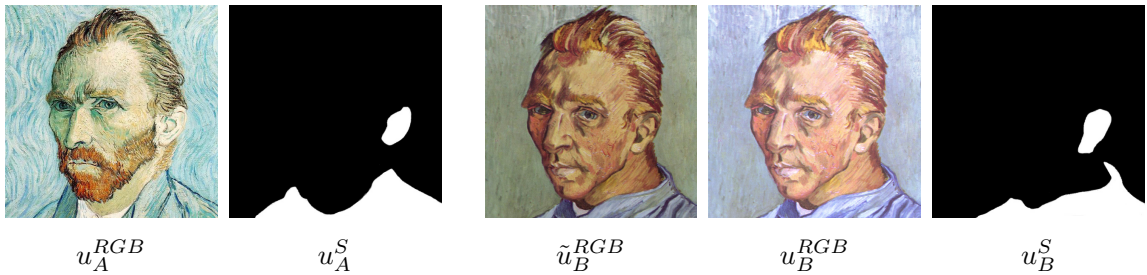
\begin{figure}
\begin{center}
\resizebox{\linewidth}{!}{
\begin{tikzpicture}
\node (label) at (0,0){\includegraphics[width=0.17\textwidth]{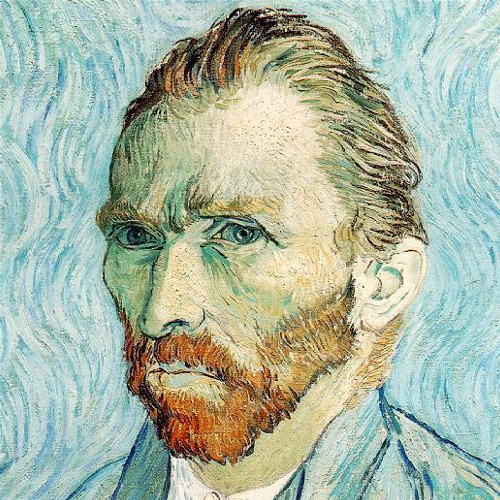}};
\node (label) at (0,-1.8) {$\image_A^{RGB}$};
\node (label) at (2.8,0){\includegraphics[width=0.17\textwidth]{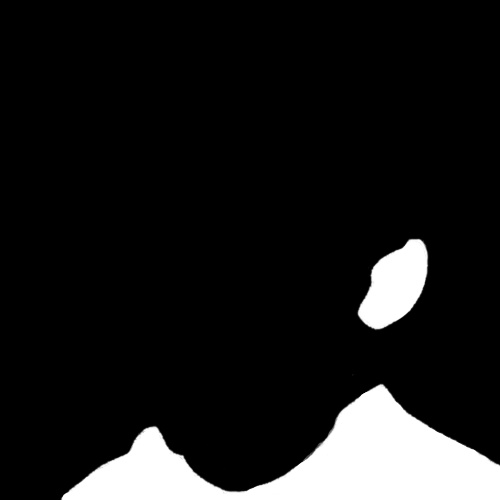}};
\node (label) at (2.8,-1.8) {$\image_A^{S}$};
\node (label) at (6,0){\includegraphics[width=0.17\textwidth]{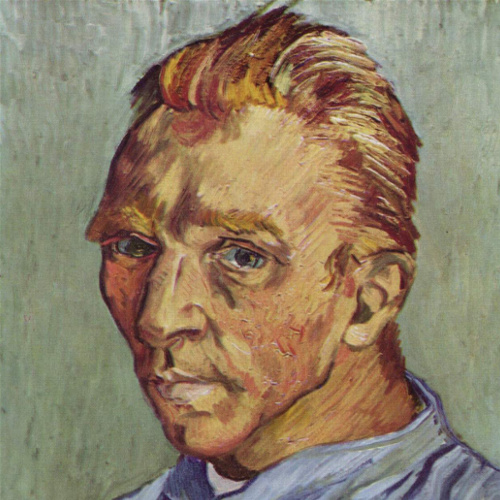}};
\node (label) at (6.0,-1.8) {$\tilde\image_B^{RGB}$};
\node (label) at (8.8,0){\includegraphics[width=0.17\textwidth]{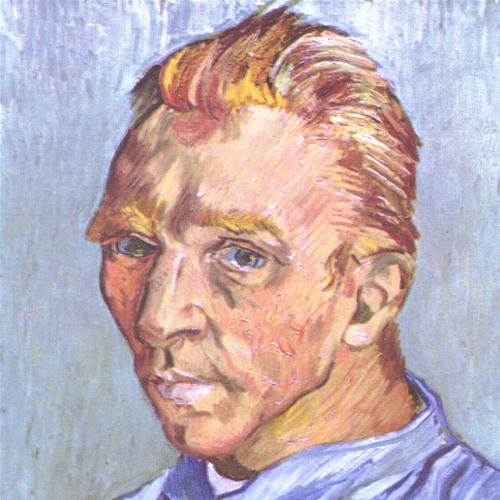}};
\node (label) at (8.8,-1.8) {$\image_B^{RGB}$};
\node (label) at (11.6,0){\includegraphics[width=0.17\textwidth]{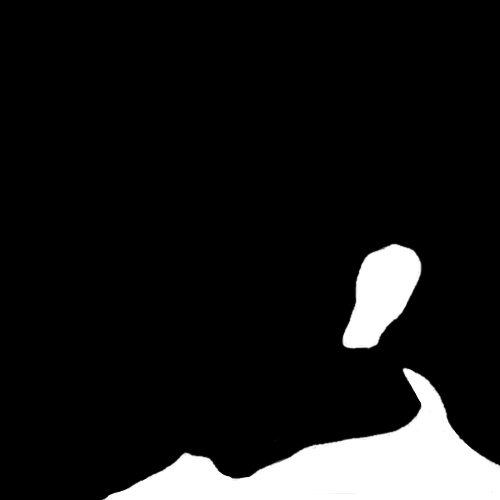}};
\node (label) at (11.6,-1.8) {$\image_B^{S}$};
\end{tikzpicture}
}
\caption{Original van Gogh self-portraits $u^{RGB}_A$, $\tilde u^{RGB}_B$ and the background modulated input image $u^{RGB}_B$ together with the associated fourth channel segmentations $u^{S}_A$ and $u^{S}_B$.}
\label{figure:vanGogh}
\end{center}
\end{figure}

\begin{figure}
\subfigure{\includegraphics[width=0.32\textwidth]{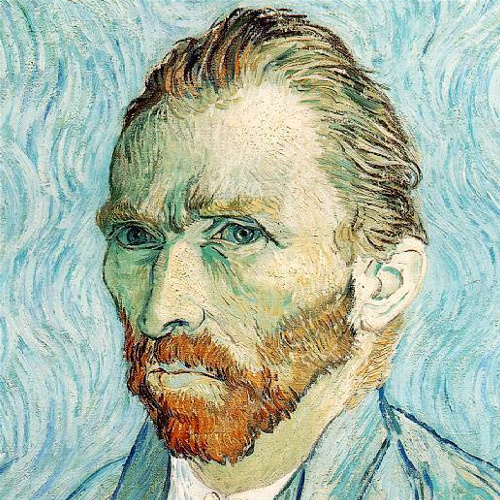}}
\hfill
\subfigure{\includegraphics[width=0.32\textwidth]{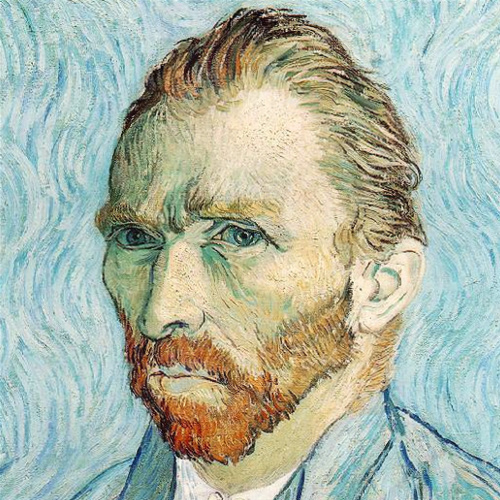}}
\hfill
\subfigure{\includegraphics[width=0.32\textwidth]{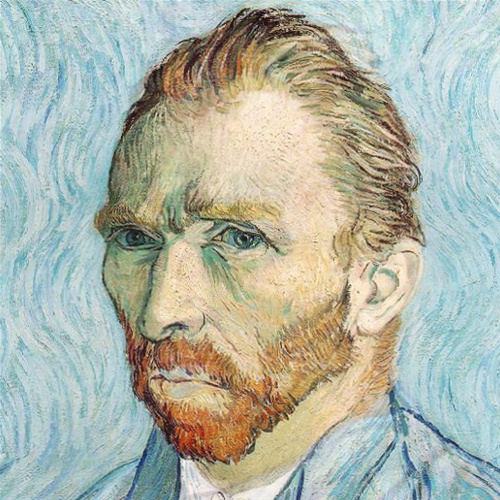}}
\subfigure{\includegraphics[width=0.32\textwidth]{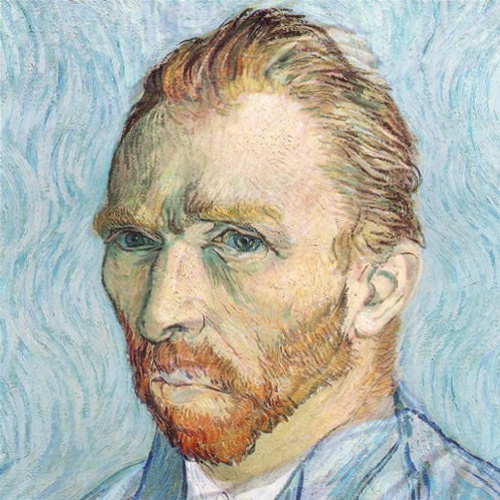}}
\hfill
\subfigure{\includegraphics[width=0.32\textwidth]{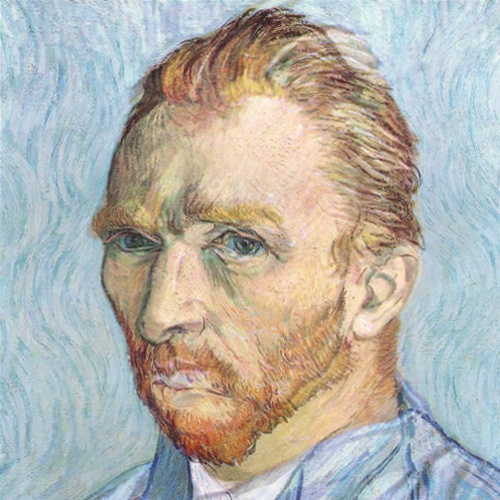}}
\hfill
\subfigure{\includegraphics[width=0.32\textwidth]{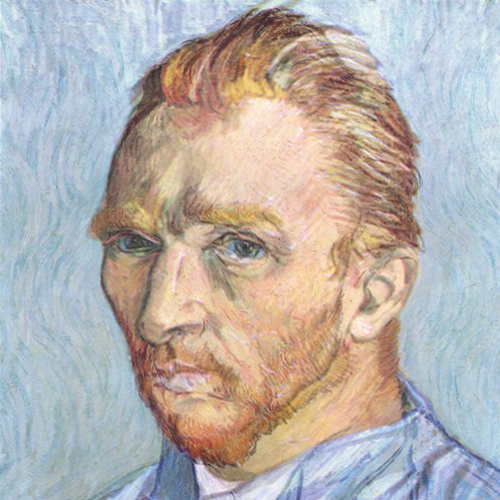}}
\subfigure{\includegraphics[width=0.32\textwidth]{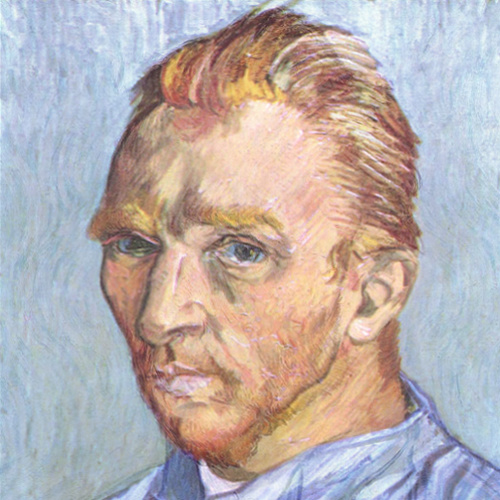}}
\hfill
\subfigure{\includegraphics[width=0.32\textwidth]{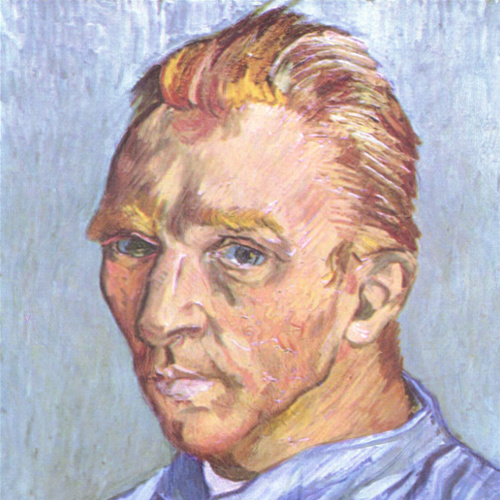}}
\hfill
\subfigure{\includegraphics[width=0.32\textwidth]{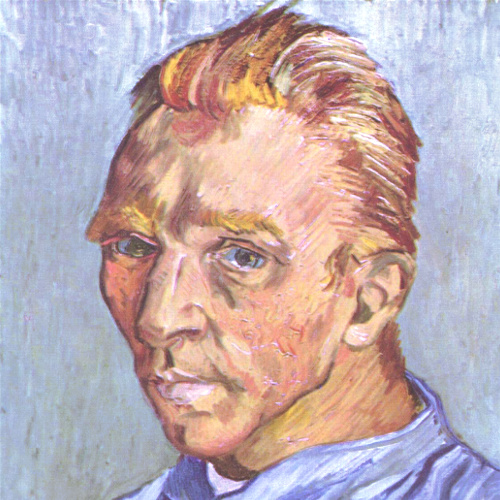}}
\hspace*{-0.01\linewidth}
\resizebox{1.01\linewidth}{!}{
\subfigure{\includegraphics[width=0.10\textwidth]{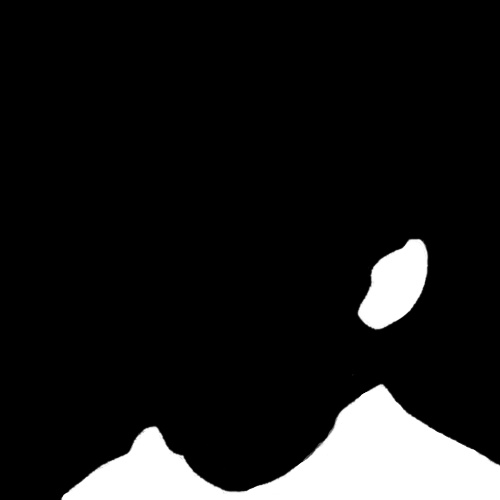}}
\hfill
\subfigure{\includegraphics[width=0.10\textwidth]{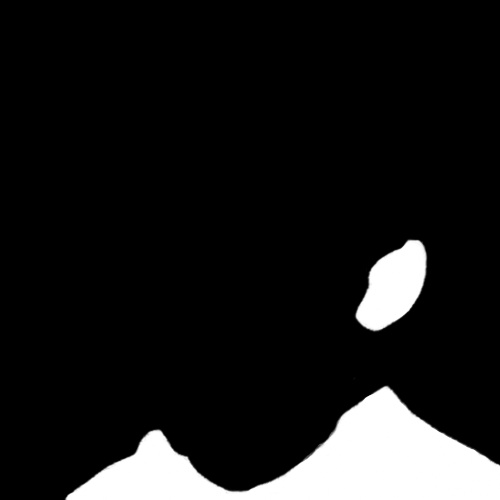}}
\hfill
\subfigure{\includegraphics[width=0.10\textwidth]{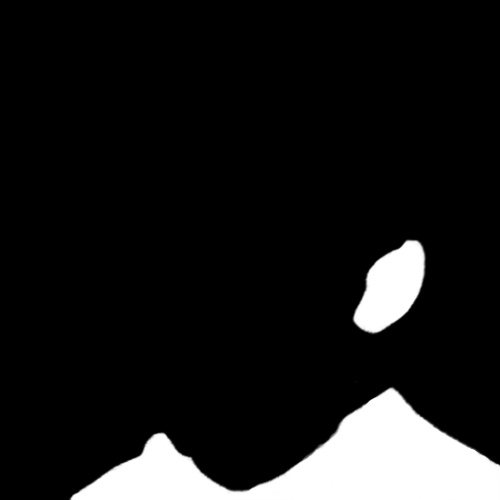}}
\hfill
\subfigure{\includegraphics[width=0.10\textwidth]{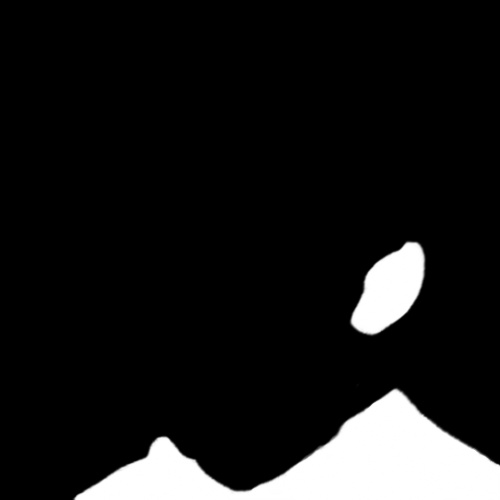}}
\hfill
\subfigure{\includegraphics[width=0.10\textwidth]{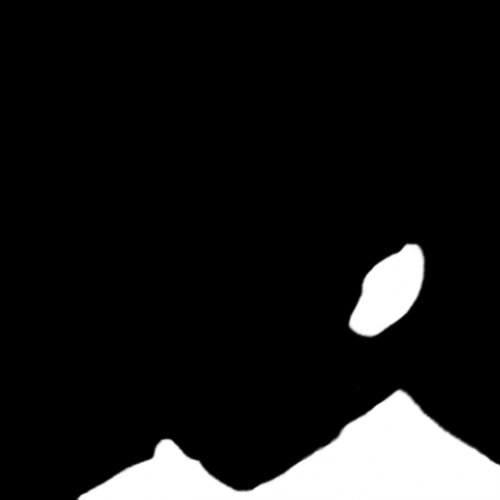}}
\hfill
\subfigure{\includegraphics[width=0.10\textwidth]{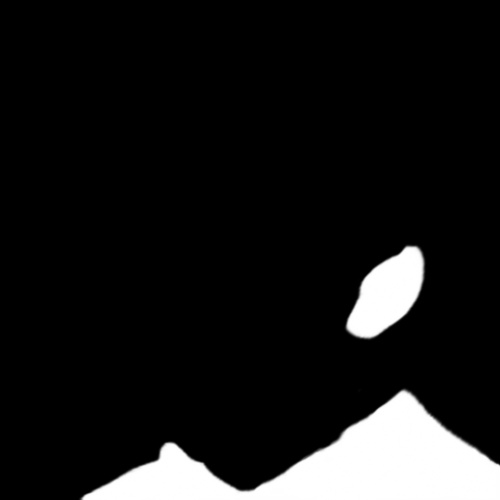}}
\hfill
\subfigure{\includegraphics[width=0.10\textwidth]{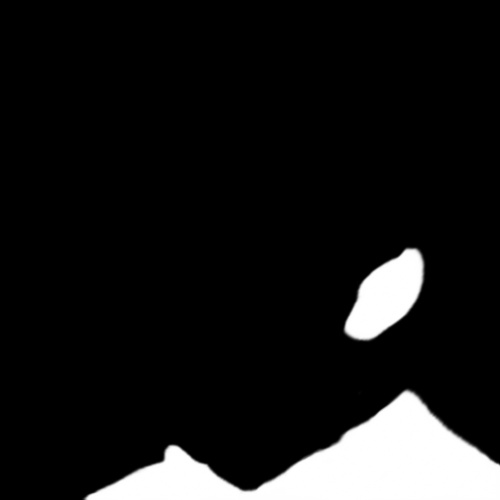}}
\hfill
\subfigure{\includegraphics[width=0.10\textwidth]{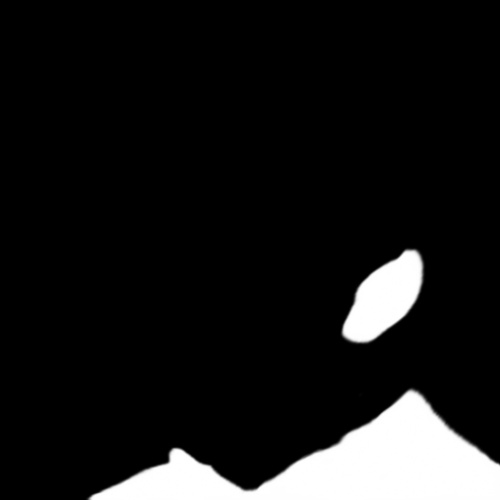}}
\hfill
\subfigure{\includegraphics[width=0.10\textwidth]{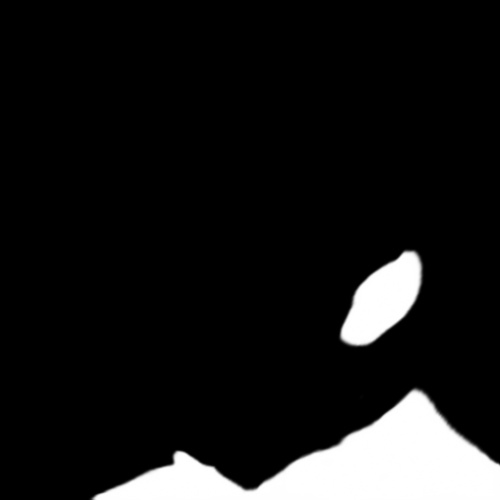}}
}
\caption{Metamorphosis between two ``van Gogh self-portraits'' using the energy \eqref{eq:WEnerDefinitionsimple}
for $K=8$ and $\delta=10^{-2}$ 
including the fourth (segmentation) channel (bottom row).}
\label{figure:vanGoghGeodesicPath}
\end{figure}

\begin{figure}
\resizebox{\linewidth}{!}{
\subfigure{\includegraphics[width=0.11\textwidth]{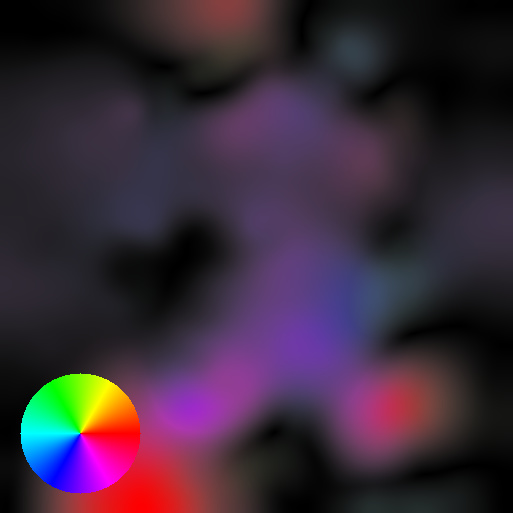}}
\hfill
\subfigure{\includegraphics[width=0.11\textwidth]{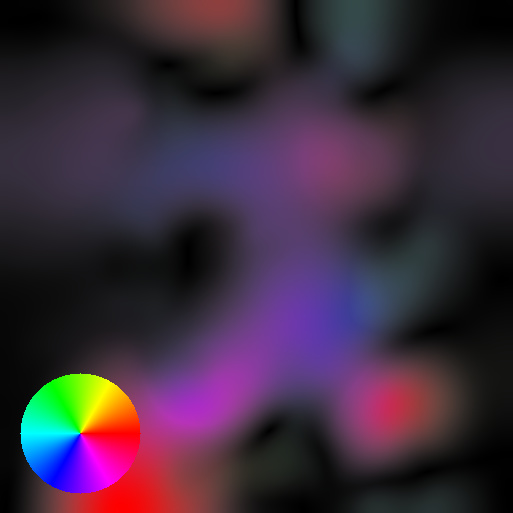}}
\hfill
\subfigure{\includegraphics[width=0.11\textwidth]{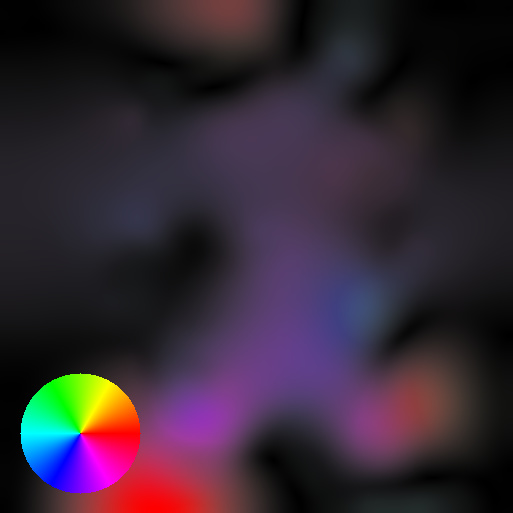}}
\hfill
\subfigure{\includegraphics[width=0.11\textwidth]{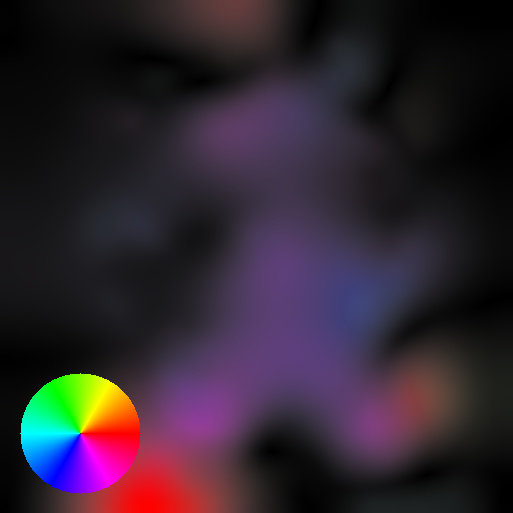}}
\hfill
\subfigure{\includegraphics[width=0.11\textwidth]{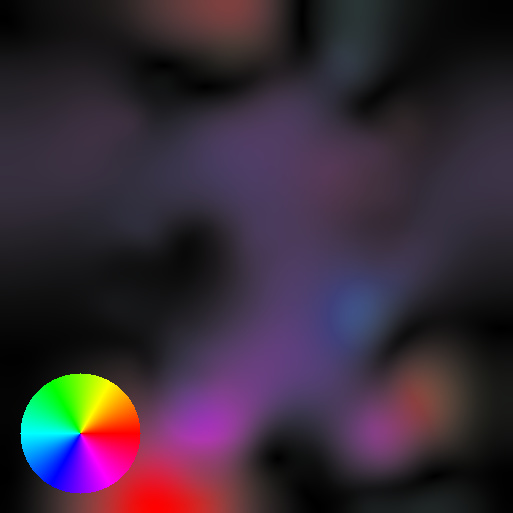}}
\hfill
\subfigure{\includegraphics[width=0.11\textwidth]{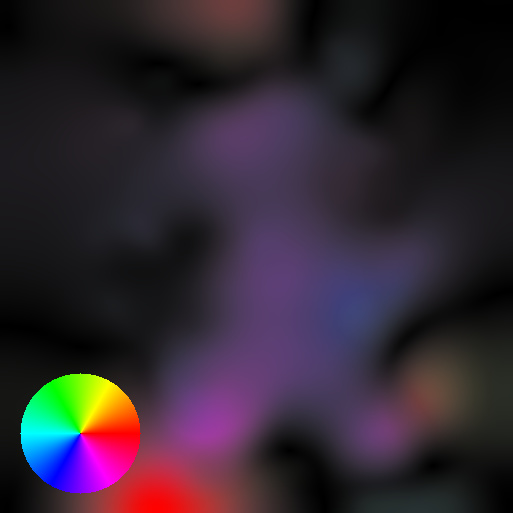}}
\hfill
\subfigure{\includegraphics[width=0.11\textwidth]{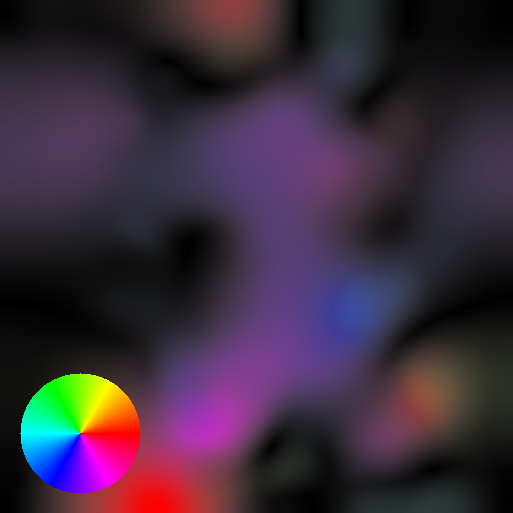}}
\hfill
\subfigure{\includegraphics[width=0.11\textwidth]{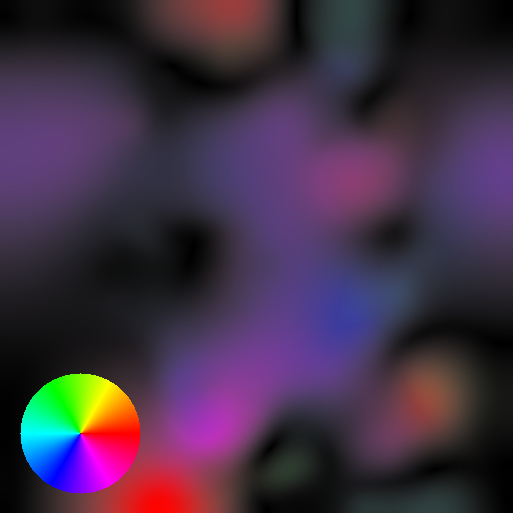}}
}
\resizebox{1.0\linewidth}{!}{
\subfigure{\includegraphics[width=0.11\textwidth]{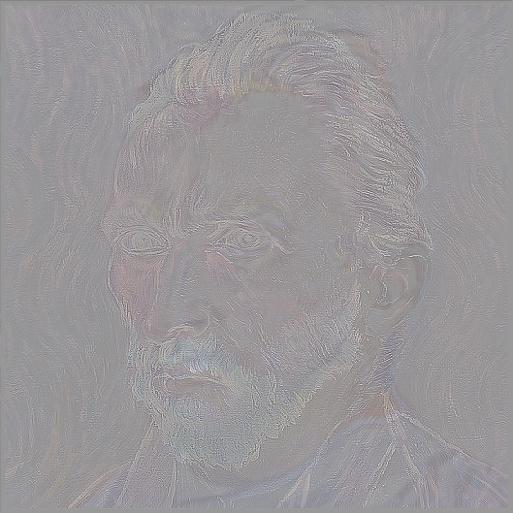}}
\hfill
\subfigure{\includegraphics[width=0.11\textwidth]{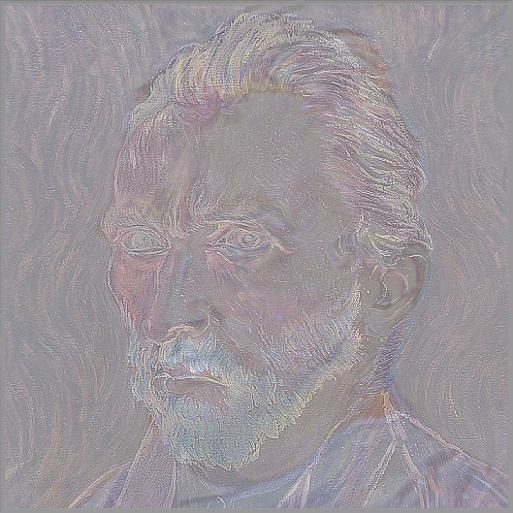}}
\hfill
\subfigure{\includegraphics[width=0.11\textwidth]{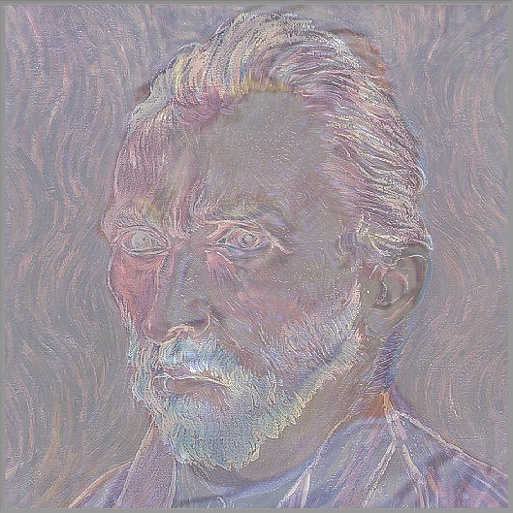}}
\hfill
\subfigure{\includegraphics[width=0.11\textwidth]{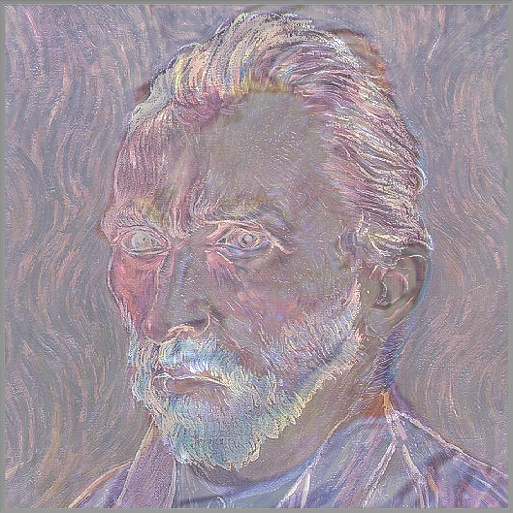}}
\hfill
\subfigure{\includegraphics[width=0.11\textwidth]{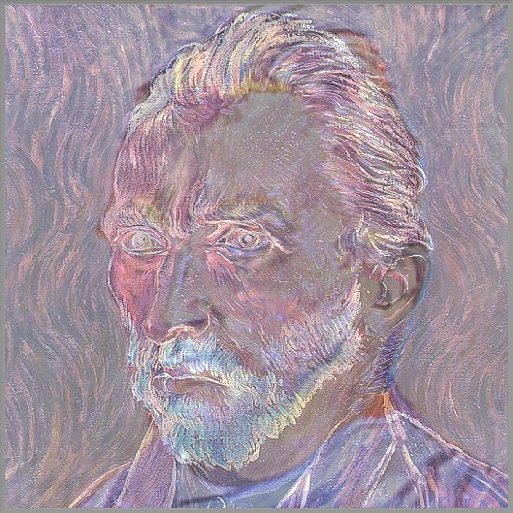}}
\hfill
\subfigure{\includegraphics[width=0.11\textwidth]{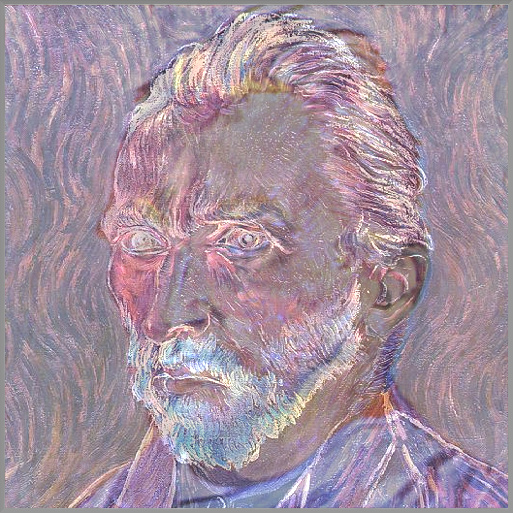}}
\hfill
\subfigure{\includegraphics[width=0.11\textwidth]{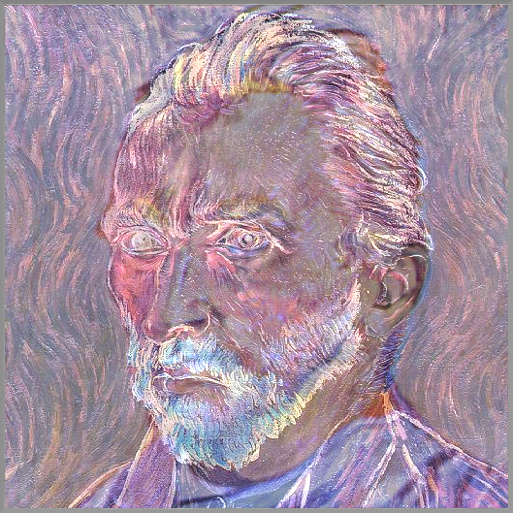}}
\hfill
\subfigure{\includegraphics[width=0.11\textwidth]{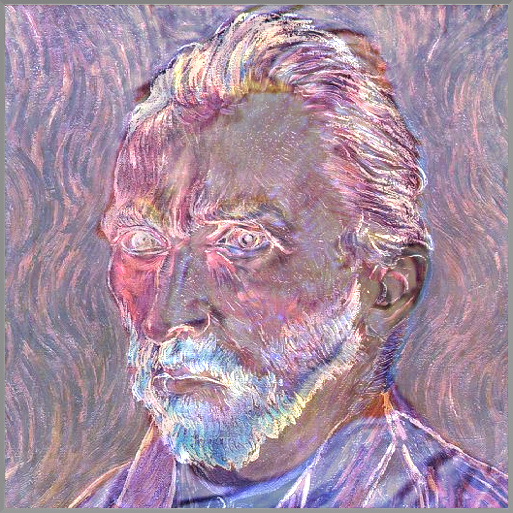}}
}
\caption{Discrete motion fields $K(\Phi_k-1)$  (first row) and accumulated weak material derivative $Z_l$ (second row) for $k= 1,\ldots, 9$. }
\label{vzComponentVanGogh}
\end{figure}

\section{Conclusions and outlook}
We have developed a robust and effective time discrete approximation for the metamorphosis approach to compute shortest paths in the space of images.
Thereby, the underlying discrete path energy is a sum of classical image matching functionals. 
The approach allows for edge type singularities in the input images. We have proven existence of 
minimizers of the discrete path energy and convergence 
of minimizing discrete paths to a continuous path, which minimizes the continuous path energy.
This analysis is based on a combination of the variational perspective of (discrete) geodesics as minimizers of the continuous \eqref{eq:DefinitionPathenergy} and discrete path energy \eqref{eq:pathenergy}, respectively,
with the continuous (\eqref{eq:ODE}, \eqref{eq:Int}) and discrete flow perspective  (\eqref{eq:discreteODE}, \eqref{eq:discreteInt}). 
In particular, this combination is the basis of a compensated compactness argument for the weak material derivative.
Indeed, using the flow perspective \eqref{eq:Int}, we are able to compensate for the loss of compactness in time, when 
trying to pass to the limit in the weak definition of the discrete material derivative \eqref{eq:weakmaterial}.
Using a finite element ansatz for the spatial discretization, a numerical algorithm 
has been presented to compute discrete geodesic paths. Qualitative properties of the algorithm are discussed for three different examples including an application to multi-channel images.
Particularly interesting future research directions are 
\begin{itemize}
\item[-] the use of duality techniques in PDE constraint optimization to derive a Newton type scheme for the simultaneous optimization of the set of deformations and the set of images associated with the discrete path,
\item[-] a full-fledged discrete geodesic calculus based on the general procedure developed in \cite{RuWi12, RuWi12b} and including a discrete logarithmic map, a discrete exponential map, and a discrete parallel transport, and
\item[-] a concept for discrete geodesic regression and geometric, statistical analysis in the space of images.
\end{itemize}
Furthermore, the close connection to optimal transportation offers interesting perspectives, which should be exploited. 

\section*{Acknowledgements}
The authors acknowledge support of the Hausdorff Center for Mathematics, the Bonn International Graduate School in Mathematics and the Collaborative Research Centre 1060  funded by the German Science foundation. B. Berkels was funded in part by the Excellence Initiative of the German Federal and State Governments.

{\small
\bibliographystyle{alpha}
\bibliography{Bibtex/all,Bibtex/library,Bibtex/own}
}

\end{document}